\documentclass[reqno,10pt]{amsart}
\numberwithin{equation}{section}
\usepackage{latexsym}
\usepackage{amsmath}
\usepackage{amssymb}
\usepackage{mathrsfs}
\usepackage{graphicx,colordvi}
\usepackage{upgreek}
\usepackage{ifthen}
\usepackage{color}

\usepackage[T1]{fontenc}
\usepackage[latin1]{inputenc}

\numberwithin{equation}{section}

\newtheorem{defi}{Definition}[section]
\newtheorem{theorem}[defi]{Theorem}
\newtheorem{lemma}[defi]{Lemma}
\newtheorem{corollary}[defi]{Corollary}
\newtheorem{proposition}[defi]{Proposition}
\newtheorem{remark}[defi]{Remark}
\newtheorem{remarks}[defi]{Remarks}

\usepackage{latexsym}
\usepackage{amsmath}
\usepackage{amssymb}

\newcommand{\cD}{{\mathcal D}}

\newcommand{\cE}{{\mathcal E}}

\newcommand{\cP}{{\mathcal P}}
\newcommand{\cT}{{\mathcal T}}

\newcommand{\EE}{{\mathbb E}}

\newcommand{\N}{{\mathbb N}}

\newcommand{\R}{{\mathbb R}}

\renewcommand{\epsilon}{\varepsilon}

\newcommand{\bb}{{\bb}}

\newcommand{\la}{\langle}
\newcommand{\ra}{\rangle}

\begin{document}

\title[$H^\infty$-calculus for the surface Stokes operator]{$H^\infty$-calculus for the surface Stokes operator and applications}

%\author{Jan Pr\"uss$\dag$ }
%\address{Martin-Luther-Universit\"at Halle-Witten\-berg\\
%         Institut f\"ur Mathematik \\
%         Theodor-Lieser-Strasse 5\\
%         D-06120 Halle, Germany}
%\email{jan.pruess@mathematik.uni-halle.de}

\author{Gieri Simonett}
\address{Department of Mathematics\\
        Vanderbilt University\\
        Nashville, Tennessee\\
        USA}
\email{gieri.simonett@vanderbilt.edu}

\author{Mathias Wilke}
\address{Martin-Luther-Universit\"at Halle-Witten\-berg\\
         Institut f\"ur Mathematik \\
         Halle (Saale), Germany}
\email{mathias.wilke@mathematik.uni-halle.de}

\thanks{This work was supported by a grant from the Simons Foundation (\#426729 and \#853237, Gieri Simonett).}

\subjclass[2010]{35Q35, 35Q30, 35B40}
% 35B40 Asymptotic behavior of solutions
% 35Q30 Navier-Stokes equations
% 35Q35 PDEs in connection with fluid mechanics

 \keywords{Surface Navier-Stokes equations, surface Stokes operator, 
 $H^\infty$-calculus, critical spaces, Killing vector fields, Korn's inequality, global existence.}

\begin{abstract}
We consider a smooth, compact and embedded hypersurface $\Sigma$  without boundary  and show that the corresponding (shifted) surface Stokes operator $\omega+A_{S,\Sigma}$ admits a bounded $H^\infty$-calculus with %$H^\infty$-angle % $\phi_{A_{S,\Sigma}}^\infty<\pi/2$,  
angle smaller than $\pi/2$,
provided $\omega>0$.
As an application, we consider critical spaces for the Navier-Stokes equations on the surface $\Sigma$. 
In case $\Sigma$ is two-dimensional, we show that any solution with a divergence-free initial value
in $L_2(\Sigma, \mathsf{T}\Sigma)$ exists globally and converges exponentially fast  to an equilibrium, that is, to a Killing field.
\end{abstract}

\maketitle
%%%%%%%%%%%%%%%%%%%%%
\section{Introduction}
\noindent
Suppose $\Sigma$ is a smooth, compact, {connected}, embedded hypersurface in $\R^{d+1}$ without boundary.
We then consider the motion of an incompressible viscous fluid that completely covers $\Sigma$
and flows along $\Sigma$. The motion can be modeled by the {\em surface Navier-Stokes equations} for an incompressible viscous fluid
    \begin{equation}
    \label{NS-surface-introduction}
    \begin{aligned}
    \varrho \big(\partial_t u + \cP_\Sigma (u\cdot \nabla_\Sigma u)\big) - \cP_\Sigma\, {\rm div}_\Sigma\, \cT_\Sigma &=0 &&\text{on}\;\;\Sigma \\
    {\rm div}_\Sigma u &=0 &&\text{on} \;\; \Sigma \\
                           u(0) &= u_0  &&\text{on}\;\; \Sigma.
    \end{aligned}
    \end{equation}
 Here, the density $\varrho$ is a positive constant, $\cT_\Sigma =2\mu_s \cD_\Sigma(u) -\pi \cP_\Sigma$, and
\begin{equation*}
\label{D-Sigma}
\cD_\Sigma(u):=\frac{1}{2}\cP_\Sigma\left( \nabla_\Sigma u + [\nabla_\Sigma u]^{\sf T}\right)\cP_\Sigma
\end{equation*}
is the surface rate-of-strain tensor, with $u$ the fluid velocity and $\pi$ the pressure.
Moreover,   
$\cP_\Sigma$ denotes the orthogonal projection onto the tangent bundle ${\sf T}\Sigma$ of $\Sigma$,
${\rm div}_\Sigma$  the surface divergence, and $\nabla_\Sigma$ the surface gradient.
We refer to Chapter 2 in \cite{PrSi16} and the Appendix in \cite{PSW20} for more information concerning these objects. 

\smallskip
 The formulation \eqref {NS-surface-introduction} coincides with \cite[formula (3.2)]{JOR18}.
In that paper, the equations were derived from fundamental continuum mechanical principles.
The same equations were also derived in~\cite[formula (4.4)]{KLG17},
based on global energy principles.
We mention that the authors of  \cite{JOR18, KLG17} also consider material surfaces that
may evolve in time.

\smallskip
Existence and uniqueness of solutions to the  surface Navier-Stokes equations
\eqref{NS-surface-introduction} was established in~\cite{PSW20}. 
It was shown that the set of equilibria consists of all Killing vector fields on $\Sigma$,
and that all of these are normally stable.  
Moreover, it was shown
that~\eqref{NS-surface-introduction} can be reformulated as
\begin{equation}
\label{NS-surface-2}
\begin{aligned}
\varrho \big(\partial_t u + \cP_\Sigma (u\cdot \nabla_\Sigma u)\big)  - \mu_s( \Delta_\Sigma 
+ {\sf Ric}_\Sigma )u + \nabla_\Sigma \pi&=0 &&\text{on}\;\;\Sigma \\
{\rm div}_\Sigma u &=0 &&\text{on} \;\; \Sigma \\
                       u(0) &= u_0  &&\text{on}\;\; \Sigma,
\end{aligned}
\end{equation}
where $\Delta_\Sigma$ is the (negative) Bochner-Laplacian and ${\sf Ric}_\Sigma$ the Ricci tensor.
To be more precise, in slight abuse of the usual convention, we interpret ${\sf Ric}_\Sigma$ here as the $(1,1)$-tensor 
given in local coordinates by ${\sf Ric}^\ell_k= g^{i\ell}R_{ik}$. 
We remind that in case $d=2$, ${\sf Ric}_\Sigma u=K_\Sigma u$, with $K_\Sigma$ being the Gaussian curvature of $\Sigma$.

\smallskip
The formulation (1.2) shows that  the surface Navier-Stokes equations can be formulated by intrinsic quantities that only
depend on the geometry of the surface $\Sigma$, but not on the ambient space.
In an intrinsic formulation, the surface Navier Stokes equations~\eqref{NS-surface-2} can be stated as
\begin{equation}
\label{NS-surface-intrinsic}
\begin{aligned}
\varrho \big(\partial_t u + \nabla_u u) - \mu_s \Delta_\Sigma u- \mu_s{\sf Ric}\, u + {\sf grad }\,\pi&=0 &&\text{on}\;\;\Sigma \\
{\rm div} u &=0 &&\text{on} \;\; \Sigma \\
                       u(0) &= u_0  &&\text{on}\;\; \Sigma,
\end{aligned}
\end{equation}
where $\nabla$ is the covariant derivative (induced by the Levi-Civita connection of $\Sigma$),
and ${\sf grad}\,\pi=\nabla_\Sigma \pi$.
An inspection of the proofs then shows that all the results in \cite{PSW20}, and the results of this paper, 
are also valid for \eqref{NS-surface-intrinsic}  for any smooth Riemannian manifold $\Sigma$ without boundary.
We remind that the Bochner Laplacian $\Delta_\Sigma$ is related to the Hodge Laplacian by the formula
$\Delta_\Sigma = \Delta_H +{\sf Ric},$
with the usual identification of vector fields and one-forms by means of lowering or rising indices.

\smallskip
We would like to point out that several formulations for the
surface Navier-Stokes equations have been considered in the literature,
see~\cite{CCD17} for a comprehensive discussion, and also~\cite[Section 3.2]{JOR18}.
It has been advocated in \cite[Note added to Proof]{EbMa70}, and also in the more recent publications \cite{CCD17,SaTu20}, 
that the surface Navier-Stokes equations on a Riemannian manifold ought to be modeled by the system~\eqref{NS-surface-intrinsic}.

\smallskip
In Theorem~\ref{thm:Hinftycalc}, we show that {\em the surface Stokes operator}
\begin{equation*}
A_{S,\Sigma}u:=-2 \mu_sP_{H,\Sigma} \mathcal{P}_\Sigma{\rm div}_\Sigma\mathcal{D}_\Sigma(u)
= -\mu_sP_{H,\Sigma}\big(\Delta_\Sigma +{\sf Ric}_\Sigma \big),
\end{equation*}
with $P_{\Sigma, H}$ being the {\em surface Helmholtz projection},
has the property that 
$(\omega + A_{S,\Sigma}) $ admits a bounded $H^\infty$-calculus in $L_{q,\sigma}(\Sigma, {\sf T}\Sigma)$ with  $H^\infty$-angle $\phi_{A_{S,\Sigma}}^\infty<\pi/2$, 
provided
$\omega$ is larger than the spectral bound of $-A_{S,\Sigma}$ and $1<q<\infty$.

For the Stokes operator on {\em domains in Euclidean space} under various boundary conditions, 
the existence of an $H^\infty$-calculus (or the related property of bounded imaginary powers) 
has been obtained by 
Giga~\cite{Gig85}, Abels~\cite{Abe02}, Noll and Saal~\cite{NoSa03}, Saal~\cite{Sal06},  Pr\"uss and Wilke~\cite{PrWi18},
and Pr\"uss~\cite{Pru18}.
We also refer to the survey article by Hieber and Saal~\cite{HiSa18} for additional  references and information concerning the Stokes operator
on domains in Euclidean space.

\smallskip
Having established the existence of a bounded $H^\infty$-calculus allows us to employ the results in 
Pr\"uss, Simonett, and Wilke \cite{ PSW18, PrWi17} to establish existence and uniqueness of solutions to 
the system \eqref{NS-surface-introduction}, or \eqref{NS-surface-intrinsic}, 
for initial values $u_0$ in the {\em critical spaces} $B^{d/q-1}_{qp,\sigma}(\Sigma, {\sf T}\Sigma)$,
see Theorem~\ref{thm:strong-solution} and Theorem~\ref{thm:weakSrurfaceStokes}.

In particular, our results imply existence and uniqueness of solutions for initial values 
$u_0\in L_{q,\sigma}(\Sigma, {\sf T}\Sigma)$ for $q=d$, see Corollary~\ref{cor:Lq}.
Hence, the celebrated result of Kato~\cite{Kat84} is also valid for the surface Navier-Stokes equations.

For $d=2$, we show in Theorem~\ref{thm:convergence}
that {\em any} solution to \eqref{NS-surface-introduction} with initial value $u_0\in L_{2,\sigma}(\Sigma, {\sf T}\Sigma)$ 
 exists globally and converges exponentially fast  to an equilibrium, that is, to a Killing field.
 The proof is based on an abstract result in \cite{PSW18}, Korn's inequality (established in Theorem~\ref{thm:Korn}),
 and an energy estimate. Moreover, in Remark~\ref{rem:convergence} we show that in case $d\ge 3$, any \emph{global} solution converges to an equilibrium.

\smallskip
We refer to \cite{PSW18, PrWi18} for background information on critical spaces and for a discussion
of the existing literature concerning critical spaces for the Navier-Stokes equations (and other equations)
for domains in Euclidean space.

\smallskip
We would now like to briefly compare the results of this paper with previous results by other authors.
Existence and uniqueness of solutions for the Navier-Stokes equations~\eqref{NS-surface-intrinsic}
for initial data in Morrey and Besov-Morrey spaces
was established  by Taylor  \cite{Tay92} and Mazzucato~\cite{Maz03}, respectively; see
also ~\cite{CCD17} for a comprehensive list of references.
The authors in~\cite{Tay92, Maz03} employ  techniques of pseudo-differential operators and they make
use of the property that the Hodge Laplacian commutes with the Helmholtz projection.
In case $d=2$, global existence is proved in \cite[Proposition 6.5]{Tay92}, but that result does not establish convergence of solutions.

\smallskip
The Boussinesq-Scriven surface stress tensor $\cT_\Sigma$ has also been employed
in the situation where two incompressible fluids which are separated by a free surface, where
surface viscosity (accounting for internal friction within the interface) is included in the model,
see for instance \cite{BoPr10, PrSi16}.
Finally, we mention~\cite{JOR18, OQRY18, ReZh13,ReVo18} and the references 
contained therein for interesting numerical investigations.

\medskip
\medskip\noindent
{\bf Notation:}
We now introduce some notation and some auxiliary results that will be used in the sequel.
It follows from the considerations in \cite[Lemma A.1 and Remarks A.3]{PSW20} that
\begin{equation}
\label{notation}
(\cP_\Sigma\nabla_\Sigma u)^{\sf T}=\nabla u\quad\text{and}\quad
 \cP_\Sigma(u\cdot \nabla_\Sigma v)= \nabla_u v
\end{equation}
for tangential vector fields $u,v$, where $\nabla$ denotes the covariant derivative (or the Levi-Civita connection) on $\Sigma$.
In the following, we will occasionally take the liberty to use the shorter notation
$\nabla_u v$ and $\nabla u$ mentioned in \eqref{notation}.
We recall that for sufficiently smooth vectors fields $u$, say $u \in C^1(\Sigma,{\sf T}\Sigma)$, one has 
%$\nabla_u v \in C(\Sigma,{\sf T}\Sigma)$ and 
$\nabla u\in C(\Sigma, {\sf T}^1_1\Sigma)$, the space of all  $(1,1)$-tensors on $\Sigma$.
%${\sf T}^1_1\Sigma={\sf T}\Sigma\otimes {\sf T}^*\Sigma$ 
As the Levi-Civita connection $\nabla$ is a metric connection, we have
\begin{equation}
\label{metric-connection}
\nabla_w(u|v)= (\nabla_wu|v) + (u| \nabla_w u)
\end{equation}
for tangential vector fields $u,v,w$ on $\Sigma$, where $(u|v):=g(u,v)$ is the Riemannian metric 
(induced by the Euclidean inner product of $\R^{d+1}$ in case $\Sigma$ is embedded in~$\R^{d+1}$).
Occasionally, we also write
${\sf grad}\,\varphi$ in lieu of  $\nabla_\Sigma \varphi$
for scalar functions~$\varphi$.

\smallskip
\noindent
We use the notation %$(u|v)_\Sigma = \int_\Sigma (u|v)\,d\Sigma$
\begin{equation*}
(u|v)_\Sigma = \int_\Sigma (u|v)\,d\Sigma,
\end{equation*}
whenever the right hand side exists, say for $u\in L_q(\Sigma, {\sf T}\Sigma)$ and $v\in L_{q'}(\Sigma, {\sf T}\Sigma)$,
where $1/q+1/q'=1$.
For $k\in \N$ and $q\in (1,\infty)$, the space $H^k_q(\Sigma, {\sf T}\Sigma)$ is defined as  the completion of 
$ C^\infty(\Sigma, {\sf T}\Sigma)$, the space of all smooth vector fields, in $L_{1,{\rm loc}}(\Sigma, {\sf T}\Sigma)$
with respect to the norm
\begin{equation*}
|u|_{H^k_q(\Sigma)}=
\left(\sum_{i=0}^k |\nabla^i u|^q_{L_q(\Sigma)}\right)^{1/q}.
\end{equation*}
%    The inner product in  $H^1_2(\Sigma, {\sf T}\Sigma)$ is then given by
%    \begin{equation*}
%    \label{H1-inner-product}
%    (u|v)_{H^1_2(\Sigma, {\sf T} \Sigma)}= \int_\Sigma \big( (u | v) + (\nabla u | \nabla v)\big)\,d\Sigma,
%    \end{equation*}
%    where 
%    $(\nabla u | \nabla v)={\sf tr}(\nabla u\, [\nabla v]^{\sf T})$.
%    %= u^i_{|j} v^j_{|i}, with  $u^i_{|j}$  denoting the covariant derivatives of $u$.
The Bessel potential spaces $H^s_q(\Sigma, {\sf T}\Sigma)$ and the Besov spaces $B^s_{qp}(\Sigma, {\sf T}\Sigma)$  can then be defined
through interpolation, see for instance \cite[Section 7]{Ama13}.
It is well-known that these spaces can be given equivalent norms by means of local coordinates, see for instance \cite[Theorem 7.3]{Ama13} (for the more general context of singular and uniformly regular manifolds).
%%%%%%%%%%%%
\section{The surface Stokes operator}\label{section:Surface-Stokes}
%%%%%%%%%%%%

In \cite[Corollary 3.4]{PSW20} we showed that there exists a number $\omega_0>0$ such that for $\omega>\omega_0$, the system
\begin{equation}
\label{Stokes-surface-full}
\begin{aligned}
\partial_t u+\omega u - 2\mu_s \mathcal{P}_\Sigma{\rm div}_\Sigma\mathcal{D}_\Sigma(u) + \nabla_\Sigma \pi&=f &&\text{on}\;\;\Sigma \\
{\rm div}_\Sigma u &=0 &&\text{on} \;\; \Sigma \\
                       u(0) &= u_0  &&\text{on}\;\; \Sigma
\end{aligned}
\end{equation}
admits a unique solution
$$u\in H_{p,\mu}^1(\R_+;L_q(\Sigma,\mathsf{T}\Sigma))\cap L_{p,\mu}(\R_+;H_q^2(\Sigma,\mathsf{T}\Sigma)),\quad \pi\in L_{p,\mu}(\R_+;\dot{H}_q^1(\Sigma)),$$
if and only if
$$f\in L_{p,\mu}(\R_+;L_q(\Sigma,\mathsf{T}\Sigma)),\ u_0\in B_{qp}^{2\mu-2/p}(\Sigma,\mathsf{T}\Sigma)\quad\text{and}\quad {\rm div}_\Sigma u_0=0.$$
Moreover, the solution $(u,\pi)$ depends continuously on the given data $(f,u_0)$ in the corresponding spaces.

Let $P_{H,\Sigma}$ denote the \emph{surface Helmholtz projection}, defined by
$$P_{H,\Sigma} v:=v-\nabla_\Sigma\psi_v,\quad v\in L_q(\Sigma,\mathsf{T}\Sigma),$$
where $\nabla_\Sigma\psi_v\in L_q(\Sigma,\mathsf{T}\Sigma)$ is the unique solution of
$$(\nabla_\Sigma\psi_v|\nabla_\Sigma\phi)_\Sigma=(v|\nabla_\Sigma\phi)_\Sigma,\quad \phi\in\dot{H}_{q'}^1(\Sigma),$$
{thanks to Lemma \ref{lem:auxellprb1}.}
We note that $(P_{H,\Sigma}u|v)_\Sigma=(u|P_{H,\Sigma} v)_\Sigma$ for all $u\in L_q(\Sigma,\mathsf{T}\Sigma)$, $v\in L_{q'}(\Sigma,\mathsf{T}\Sigma)$, which follows directly from the definition of $P_{H,\Sigma}$ (and for smooth functions from the
surface divergence theorem).
Indeed,
\begin{align*}
(P_{H,\Sigma} u|v)_\Sigma&=(u-\nabla_\Sigma\psi_u|v)_\Sigma=(u|v)_\Sigma-(\nabla_\Sigma\psi_u|v)_\Sigma\\
&=(u|v)_\Sigma-(\nabla_\Sigma\psi_v|\nabla_\Sigma\psi_u)_\Sigma\\
&=(u|v)_\Sigma-(u|\nabla_\Sigma\psi_v)_\Sigma\\
&=(u|v-\nabla_\Sigma\psi_v)_\Sigma=(u|P_{H,\Sigma}v)_\Sigma,
\end{align*}
as $\psi_u\in \dot{H}_q^1(\Sigma)$, $\psi_v\in \dot{H}_{q'}^1(\Sigma)$.
Let
$$X_0:=L_{q,\sigma}(\Sigma,\mathsf{T}\Sigma):=P_{H,\Sigma}L_q(\Sigma,\mathsf{T}\Sigma),\quad X_1:=H_{q,\sigma}^2(\Sigma,\mathsf{T}\Sigma):=H_{q}^2(\Sigma,\mathsf{T}\Sigma)\cap X_0.$$
\noindent
The \emph{surface Stokes operator} is defined by
\begin{equation}\label{eq:StokesOp}
A_{S,\Sigma}u:=-2 \mu_sP_{H,\Sigma} \mathcal{P}_\Sigma{\rm div}_\Sigma\mathcal{D}_\Sigma(u),\quad u\in D(A_{S,\Sigma}):=X_1.
\end{equation}

Making use of the projection $P_{H,\Sigma}$, \eqref{Stokes-surface-full} with ${\rm div}_\Sigma f=0$ is equivalent to the equation
\begin{equation}
\label{Abstract-Stokes}
\partial_t u+\omega u+A_{S,\Sigma}u=f,\quad t>0,\quad u(0)=u_0.
\end{equation}
Indeed, if $(u,\pi)$ is a solution to \eqref{Stokes-surface-full}, then $u=P_{H,\Sigma}u$ solves \eqref{Abstract-Stokes} as can be seen by applying $P_{H,\Sigma}$ to the first equation in \eqref{Stokes-surface-full}. Conversely, let $u$ be a solution of \eqref{Abstract-Stokes}. Then, by definition of $P_{H,\Sigma}$,
$$A_{S,\Sigma}u=-2 \mu_sP_{H,\Sigma} \mathcal{P}_\Sigma{\rm div}_\Sigma\mathcal{D}_\Sigma(u)=-2 \mu_s\mathcal{P}_\Sigma{\rm div}_\Sigma\mathcal{D}_\Sigma(u)+2\mu_s\nabla_\Sigma\psi_v,$$
where $\psi_v\in\dot{H}_q^1(\Sigma)$ solves
$$(\nabla_\Sigma\psi_v|\nabla_\Sigma\phi)_\Sigma=(v|\nabla_\Sigma\phi)_\Sigma ,\quad \phi\in \dot{H}_{q'}^1(\Sigma),$$
with $v:= \mathcal{P}_\Sigma{\rm div}_\Sigma\mathcal{D}_\Sigma(u)\in L_q(\Sigma)$.
Defining $\pi:=2\mu_s\psi_v$, we see that $(u,\pi)$ is a solution of \eqref{Stokes-surface-full}.

In particular, the operator $A_{S,\Sigma}$ has $L_{p}$-maximal regularity, hence $-(\omega+A_{S,\Sigma})$ generates an analytic $C_0$-semigroup in $X_0$ (see for instance \cite[Proposition 3.5.2]{PrSi16}) which is exponentially stable provided $\omega>s(-A_{S,\Sigma})$, where $s(\cdot)$ denotes the spectral bound of $-A_{S,\Sigma}$. This readily implies that the operator $\omega+A_{S,\Sigma}$ is sectorial with spectral angle $\phi_{\omega+A_{S,\Sigma}}<\pi/2$.

%%%%%%%%%%%%%%%%%%%%%%
\section{$H^\infty$-calculus}\label{sec:Hinfty}
%%%%%%%%%%%%%%%%%%%%%%
\noindent

In this section, we are going to prove the following result.
%%%%%%%%%%
\begin{theorem}
\label{thm:Hinftycalc}
Let $q\in (1,\infty)$ and $\Sigma$ be a smooth, compact, {connected}, embedded hypersurface in $\R^{d+1}$ without boundary. Let $A_{S,\Sigma}$ be the surface Stokes operator in $L_{q,\sigma}(\Sigma,\mathsf{T}\Sigma)$ defined in \eqref{eq:StokesOp}.

Then $\omega+A_{S,\Sigma}$ admits a bounded $H^\infty$-calculus with $H^\infty$-angle $\phi_{A_{S,\Sigma}}^\infty<\pi/2$ for each $\omega>s(-A_{S,\Sigma})$, the spectral bound of $-A_{S,\Sigma}$.
\end{theorem}
%%%%%%%%%%%
\begin{remark}
Let $\cE$ denote the set of equilibria. We have shown in \cite{PSW20}, Proposition~4.1, that $s(-A_{S,\Sigma})=0$ provided $\cE\neq\{0\}$ und $s(-A_{S,\Sigma})<0$ if $\cE=\{0\}$.

Therefore, it follows from Theorem \ref{thm:Hinftycalc} that for each $\omega>0$ the operator $\omega+A_{S,\Sigma}$ admits a bounded $H^\infty$-calculus with $H^\infty$-angle $\phi_{A_{S,\Sigma}}^\infty<\pi/2$. In case $\cE=\{0\}$ one may set $\omega=0$.
\end{remark}

\subsection{Resolvent and pressure estimates}

We consider the following resolvent problem
\begin{equation}
\label{Stokes-surface-1}
\begin{aligned}
\lambda u+(\omega - 2\mu_s \mathcal{P}_\Sigma{\rm div}_\Sigma\mathcal{D}_\Sigma)u + \nabla_\Sigma \pi&=f &&\text{on}\;\;\Sigma \\
{\rm div}_\Sigma u &=0 &&\text{on} \;\; \Sigma,
\end{aligned}
\end{equation}
where $\omega>s(-A_{S,\Sigma})$ and $\lambda\in\Sigma_{\pi-\phi}$, $\phi>\phi_{\omega+A_{S,\Sigma}}$. By sectoriality of the operator $\omega+A_{S,\Sigma}$ (and since $0\in\rho(\omega+A_{S,\Sigma})$), it follows that for given $f\in L_q(\Sigma,\mathsf{T}\Sigma)$ there exists a unique solution
$$u\in H_q^2(\Sigma,\mathsf{T}\Sigma),\quad\pi\in \dot{H}_q^1(\Sigma)$$
of \eqref{Stokes-surface-1}
and there is a constant $C>0$ such that
\begin{equation}\label{eq:sectestimate}
(|\lambda|+1)|u|_{L_q(\Sigma)}+|u|_{H_q^2(\Sigma)}+|\nabla\pi|_{L_q(\Sigma)}\le C|f|_{L_q(\Sigma)},
\end{equation}
for all $\lambda\in\Sigma_{\pi-\phi}$. Note that without loss of generality, we may assume that $P_0\pi=\pi$, where
$$P_0v:=v-\frac{1}{|\Sigma|}\int_\Sigma v \,d\Sigma$$
for $v\in L_1(\Sigma)$.
Furthermore, if $\operatorname{div}_\Sigma f=0$, the pressure $\pi$ satisfies the estimate
\begin{equation}\label{eq:pressure_est}
|\pi|_{L_q(\Sigma)}\le C|u|_{H_q^1(\Sigma)}
\end{equation}
for some constant $C>0$. {The proof of the estimate \eqref{eq:pressure_est} follows exactly the lines of the proof
of \cite[Proposition 3.3]{PSW20}}.

\subsection{Localization}

By compactness of $\Sigma$, there exists a family of charts $\{(U_k,\varphi_k):k\in\{1,\ldots,N\}\}$ such that $\{U_k\}_{k=1}^N$ is an open covering of $\Sigma$. Let $\{\psi_k^2\}_{k=1}^N\subset C^\infty(\Sigma)$ be a partition of unity subordinate to the open covering $\{U_k\}_{k=1}^N$. Note that without loss of generality, we may assume that $\varphi_k(U_k)=B_{\mathbb{R}^d}(0,r)$. We call $\{(U_k,\varphi_k,\psi_k):k\in\{1,\ldots,N\}\}$ a \emph{localization system} for $\Sigma$.

Let $\{\tau_{(k)j}(p)\}_{j=1}^d$ denote a local basis of the tangent space $\mathsf{T}_p\Sigma$ of $\Sigma$ at $p\in U_k$ and denote by $\{\tau_{(k)}^j(p)\}_{j=1}^d$ the corresponding dual basis of the cotangent space ${\sf T}^*_p\Sigma$ at $p\in U_k$. Accordingly, we define $g_{(k)}^{ij}=(\tau_{(k)}^i|\tau_{(k)}^j)$ and $g_{(k)ij}$ is defined in a very similar way, see the Appendix in \cite{PSW20}. Then, with $\bar{u}=u\circ\varphi_k^{-1}$, $\bar{\pi}=\pi\circ\varphi_k^{-1}$ and so on, the system \eqref{Stokes-surface-1} with respect to the local charts $(U_k,\varphi_k)$, $k\in\{1,\ldots,N\}$, reads as follows.
\begin{equation}
\label{Stokes-local-coordinates}
\begin{aligned}
\lambda  \bar{u}_{(k)}^\ell+ (\omega- \mu_s \bar{g}_{(k)}^{ij}\partial_i\partial_j)\bar{u}_{(k)}^\ell + \bar{g}_{(k)}^{i\ell}\partial_i\bar{\pi}_{(k)}&=\bar{f}_{(k)}^\ell+F_{(k)}^\ell(\bar{u},\bar{\pi}) &&\text{in}\;\;\mathbb{R}^d \\
\partial_i \bar{u}_{(k)}^i &=H_{(k)}(\bar{u}) &&\text{in} \;\; \mathbb{R}^d,
\end{aligned}
\end{equation}
where
\begin{equation*} %\label{Def-Functions}
\bar{u}_{(k)}^\ell=(\bar{u}\bar{\psi}_k|\bar{\tau}_{(k)}^\ell),\ \bar{\pi}_{(k)}=\bar{\pi}\bar{\psi}_k,
\end{equation*}
$$\bar{f}_{(k)}^\ell=(\bar{f}\bar{\psi}_k|\bar{\tau}_{(k)}^\ell),\quad F_{(k)}^\ell(\bar{u},\bar{\pi})=\bar{\pi} \bar{g}_{(k)}^{i\ell}\partial_i\bar{\psi}_k+(B_{(k)}\bar{u}|\bar{\tau}_{(k)}^\ell),$$
$\ell\in\{1,\ldots,d\}$, $B_{(k)}$ collects all terms of order at most one and
$$H_{(k)}(\bar{u})=\bar{u}^i\partial_i\bar{\psi}_k-\bar{u}_{(k)}^j (\bar{\tau}_{(k)}^i|\partial_i\bar{\tau}_{(k)j}).$$
Here, upon translation and rotation, $\bar{g}_{(k)}^{ij}(0)=\delta_j^i$ and the coefficients have been extended in such a way that $\bar{g}_{(k)}^{ij}\in W_\infty^2(\R^d)$ and $|\bar{g}_{(k)}^{ij}-\delta_j^i|_{L_\infty(\mathbb{R}^d)}\le\eta$, where $\eta>0$ can be made as small as we wish, by decreasing the radius $r>0$ of the ball $B_{\mathbb{R}^d}(0,r)$.

In order to handle system \eqref{Stokes-local-coordinates}, we define vectors in $\R^d$ as follows:
$$\bar{u}_{(k)}:=(\bar{u}_{(k)}^1,\ldots,\bar{u}_{(k)}^d),\quad \bar{f}_{(k)}:=(\bar{f}_{(k)}^1,\ldots,\bar{f}_{(k)}^d)$$
and
$${F}_{(k)}(\bar{u},\bar{\pi}):=({F}_{(k)}^1(\bar{u},\bar{\pi}),\ldots,{F}_{(k)}^d(\bar{u},\bar{\pi})).$$
Moreover, we define the matrix $G_{(k)}=(\bar{g}_{(k)}^{ij})_{i,j=1}^d\in\R^{d\times d}$. With these notations, system \eqref{Stokes-local-coordinates} reads as
\begin{equation}
\label{Stokes-local-coordinates1a}
\begin{aligned}
\lambda  \bar{u}_{(k)}+(\omega- \mu_s (G_{(k)}\nabla|\nabla))\bar{u}_{(k)} + G_{(k)}\nabla\bar{\pi}_{(k)}&=\bar{f}_{(k)}+F_{(k)}(\bar{u},\bar{\pi}) &&\text{in}\;\;\mathbb{R}^d \\
{\rm div}\ \bar{u}_{(k)} &=H_{(k)}(\bar{u}) &&\text{in} \;\; \mathbb{R}^d.
\end{aligned}
\end{equation}
Let us remove the term $H_{(k)}(\bar{u})$, since it is not of lower order. For that purpose, we solve the equation ${\rm div}( G_{(k)}\nabla \phi_k)=H_{(k)}(\bar{u})$ by Lemma \ref{lem:auxellop1} to obtain a
unique solution {$\nabla\phi_k\in H_q^2(\R^d)^d$} with
\begin{equation}\label{Est-phi}
|\nabla\phi_k|_{H_q^2(\R^d)}\le C|H_{(k)}(\bar{u})|_{H_q^1(\R^d)},
\end{equation}
where $C$ is a positive constant.
For this, observe that  $H_{(k)}(\bar{u})$ is compactly supported and we have $\int_{\R^d}H_{(k)}(\bar{u})dx=0$.
Therefore, $H_{(k)}(\bar{u})$ induces a functional on $\dot{H}_{q'}^1(\R^d)$
with norm bounded by $C|H_{(k)}(\bar{u})|_{L_q(\R^d)}.$
To see this, choose $R>0$ such that ${\rm supp}(H_{(k)}(\bar{u}))\subset B(0,R)$ and let $\phi\in \dot H^1_{q'}(\R^d).$
Then we have
\begin{equation*}
\int_{\R^d} H_{(k)}(\bar{u})\phi\,dx =\int_{B(0,R)} H_{(k)}(\bar{u})(\phi -\hat \phi)\,dx ,
\end{equation*}
where $\hat \phi = |B(0,R)|^{-1} \int_{B(0,R)} \phi \,dx$. By the Poincar\'{e}-Wirtinger inequality,
\begin{equation*}
\begin{aligned}
|\int_{B(0,R)} H_{(k)}(\bar{u})(\phi -\hat \phi)\,dx |
& \le C |H_{(k)}(\bar{u})|_{L_q} |\nabla \phi|_{L_{q'}(B(0,R))} \\
& \le C  |H_{(k)}(\bar{u})|_{L_q} |\nabla \phi|_{L_{q'}(\R^d)}.
\end{aligned}
\end{equation*}
Let
$$\tilde{u}_{(k)}=\bar{u}_{(k)}-G_{(k)}\nabla\phi_k\quad\text{and}\quad\tilde{\pi}_{(k)}=\bar{\pi}_{(k)}+(\lambda+\omega)\phi_k.$$
It follows from \eqref{Stokes-local-coordinates1a} that the functions $(\tilde{u}_{(k)}, \tilde{\pi}_{(k)})$ then solve the system
\begin{equation}
\label{Stokes-local-coordinates-1}
\begin{aligned}
\lambda  \tilde{u}_{(k)}+(\omega- \mu_s (G_{(k)}\nabla|\nabla))\tilde{u}_{(k)}+ G_{(k)}\nabla  \tilde{\pi}_{(k)}
&=\bar{f}_{(k)}+\tilde{F}_{(k)}(\bar{u},\bar{\pi}) &&\text{in}\;\;\mathbb{R}^d \\
{\rm div}\ \tilde{u}_{(k)} &=0 &&\text{in} \;\; \mathbb{R}^d,
\end{aligned}
\end{equation}
where
$\tilde{F}_{(k)}(\bar{u},\bar{\pi}):=F_{(k)}(\bar{u},\bar{\pi})+\mu_s (G_{(k)}\nabla|\nabla)(G_{(k)}\nabla\phi_k).$
In order to remove the pressure term in \eqref{Stokes-local-coordinates-1} we introduce the projections $P_k^G$, defined by
$$P_k^Gv:=v-G_{(k)}\nabla\Phi_k,\quad v\in L_q(\R^d).$$
Here, $\nabla\Phi_k\in L_q(\R^d)$ is the unique solution of ${\rm div}(G_{(k)}\nabla\Phi_k)={\rm div}\, v$ in $\dot{H}_q^{-1}(\R^d)$, established  in Lemma \ref{lem:auxellop1}. It is readily seen that $P_k^Gv=v$ if ${\rm div}\, v=0$ and $P_k^G(G_{(k)}\nabla\tilde{\pi}_{(k)})=0$.
Applying the projection $P_k^G$ to equation \eqref{Stokes-local-coordinates-1} leads to
\begin{equation}
\label{Stokes-local-coordinates-2}
\begin{aligned}
\lambda  \tilde{u}_{(k)}+(\omega- \mu_s P_k^G(G_{(k)}\nabla|\nabla))\tilde{u}_{(k)}
&=P_k^G(\bar{f}_{(k)}+\tilde{F}_{(k)}(\bar{u},\bar{\pi})) &&\text{in}\;\;\mathbb{R}^d. \\
%{\rm div}\ \tilde{u}_{(k)} &=0 &&\text{in} \;\; \mathbb{R}^d.
\end{aligned}
\end{equation}
We claim that each of the operators $\omega+A_{S,k}^G:=\omega-\mu_sP_k^G(G_{(k)}\nabla|\nabla)$ in \eqref{Stokes-local-coordinates-2} admits a bounded $H^\infty$-calculus in $P_k^GL_q(\R^d)=L_{q,\sigma}(\R^d)$, provided $\omega$ is sufficiently large. To see this, we write
$$-A_{S,k}^Gu=-A_S u+\mu_s((G_{(k)}-I)\nabla|\nabla)u-\mu_sG_{(k)}\nabla\Phi_k,$$
where $A_S=-\mu_sP_H\Delta u=-\mu_s\Delta u$ is the Stokes operator in $L_{q,\sigma}(\R^d)$ and $P_H$ is the Helmholtz projection in $\R^d$.

Recall that each matrix $G_{(k)}$ is a perturbation of the identity in $\R^{d\times d}$. Therefore,
$$|\mu_s((G_{(k)}-I)\nabla|\nabla)u|_{L_q(\R^d)}\le \eta|u|_{H_q^2(\R^d)},$$
where we may choose $\eta>0$ as small as we wish.
Furthermore,
\begin{align*}
|G_{(k)}\nabla\Phi_k|_{L_q(\R^d)}\le C|(G_{(k)}\nabla|\nabla)u-\Delta u|_{{L}_q(\R^d)}\le \eta|u|_{H_q^2(\R^d)},
\end{align*}
by Lemma \ref{lem:auxellop1}, since
$${\rm div}(G_{(k)}\nabla\Phi_k)={\rm div}\, (G_{(k)}\nabla|\nabla)u={\rm div}\, (G_{(k)}\nabla|\nabla)u-\operatorname{div}\Delta u$$
as $\operatorname{div}\Delta u=\Delta\operatorname{div} u=0$ in $\R^d$. As before, we  may choose $\eta>0$ as small as we wish.

Note that the shifted Stokes operator $\omega+A_S$ admits a bounded $H^\infty$-calculus in $L_{q,\sigma}(\R^d)$ with angle $\phi_{A_S}^\infty<\pi/2$, see e.g.\  \cite[Theorem 7.1.2]{PrSi16}. By the abstract perturbation result \cite[Theorem  3.2]{DDHPV04} (see also \cite[Section 3]{Pru18}), there exists $\omega_0>0$ such that each of the operators $\omega+ A_{S,k}^G$ admits a bounded $H^\infty$-calculus in $L_{q,\sigma}(\R^d)$ if $\omega\ge \omega_0$. Moreover, for any given $\phi_0>\phi_{A_S}^\infty$ we may assume $\phi_{A_{S,k}^G}^\infty\le \phi_0$, provided that $|G_{(k)}-I|_\infty$ is sufficiently small.

This yields the following representation of the resolvent $u=(\lambda+\omega+A_{S,\Sigma})^{-1} f$.
\begin{align*}
u&=P_{H,\Sigma}\sum_{k=1}^N\psi_k^2 u=P_{H,\Sigma}\sum_{k=1}^N\psi_k(\bar{u}_{(k)}^\ell\bar{\tau}_{(k)\ell})\circ\varphi_k\\
&=P_{H,\Sigma}\sum_{k=1}^N\psi_k((\tilde{u}_{(k)}|e_\ell)\bar{\tau}_{(k)\ell})\circ\varphi_k+P_{H,\Sigma}\sum_{k=1}^N\psi_k((G_{(k)}\nabla\phi_k|e_\ell)\bar{\tau}_{(k)\ell})\circ\varphi_k\\
&=Tu+{S(\lambda)f}+R(\lambda)f,
\end{align*}
where $\{e_\ell\}_{\ell=1}^d$ is the standard basis in $\R^d$ and
\begin{equation*}
\begin{aligned}
Tu:&=P_{H,\Sigma}\sum_{k=1}^N\psi_k((G_{(k)}\nabla\phi_k{\color{red}} |e_\ell)\bar{\tau}_{(k)\ell})\circ\varphi_k, \\
S(\lambda)f:&=P_{H,\Sigma}\sum_{k=1}^N\psi_k(((\lambda+\omega+A_{S,k}^G)^{-1}P_k^G\bar{f}_{(k)}|e_\ell)\bar{\tau}_{(k)\ell})\circ\varphi_k \\
R(\lambda)f:&=P_{H,\Sigma}\sum_{k=1}^N\psi_k(((\lambda+\omega+A_{S,k}^G)^{-1}P_k^G\tilde{F}_{(k)}(\bar{u}(f),\bar{\pi}(f))|e_\ell)\bar{\tau}_{(k)\ell})\circ\varphi_k.
\end{aligned}
\end{equation*}
In a next step, we will estimate the term $R(\lambda)f$ in $L_{q}(\Sigma,\mathsf{T}\Sigma)$. To this end, observe that the operators $P_{H,\Sigma}$ and $P_k^G$ are bounded in $L_{q}(\Sigma,\mathsf{T}\Sigma)$ and $L_{q}(\R^d)$, respectively. This, together with \eqref{eq:sectestimate}, \eqref{eq:pressure_est}, \eqref{Est-phi} and the fact that each of the operators $\omega+A_{S,k}^G$ is sectorial in $L_{q,\sigma}(\R^d)$ for $\omega$ sufficiently large, yields the estimate
\begin{equation}
\label{3/2-decay}
|R(\lambda)f|_{L_q(\Sigma)}\le \frac{C}{(|\lambda|+1)^{3/2}}|f|_{L_q(\Sigma)}
\end{equation}
for some constant $C>0$. Indeed, by \eqref{eq:sectestimate}, \eqref{eq:pressure_est}, \eqref{Est-phi}, we obtain
\begin{align*}
|\tilde{F}_{(k)}(\bar{u},\bar{\pi})|_{L_q(\R^d)}&\le C\left( |{F}_{(k)}(\bar{u},\bar{\pi})|_{L_q(\R^d)}+|\nabla\phi_k|_{H_q^2(\R^d)}\right)\\
&\le C\left(|{\pi}|_{L_q(\Sigma)}+|{u}|_{H_q^1(\Sigma)}\right)\\
&\le C|{u}|_{H_q^1(\Sigma)}\\
&\le C(1+|\lambda|)^{-1/2}((1+|\lambda|)|{u}|_{L_q(\Sigma)}+|{u}|_{H_q^2(\Sigma)})\\
&\le C(1+|\lambda|)^{-1/2}|f|_{L_q(\Sigma)}.
\end{align*}
Here we have also used complex interpolation $H_q^1(\Sigma)=[L_q(\Sigma),H_q^2(\Sigma))]_{1/2}$.

Next, observe that $T:L_{q,\sigma}(\Sigma,\mathsf{T}\Sigma)\to H_q^1(\Sigma,\mathsf{T}\Sigma)\cap L_{q,\sigma}(\Sigma,\mathsf{T}\Sigma)$,  since $G_{(k)}\in W_\infty^2(\R^d)^{d\times d}$ and $\nabla\phi_{k}\in H_q^1(\R^d)^d$ for each $k\in\{1,\ldots,N\}$, hence $T$ is compact in $L_{q,\sigma}(\Sigma,\mathsf{T}\Sigma)$. Consequently, $I-T$ is a Fredholm operator with index 0. In particular, $\ker(I-T)$ is finite dimensional and the range ${\rm ran}(I-T)$ is closed in $L_q(\Sigma)$. Let $\{v_1,\ldots,v_m\}$ be an orthonormal basis of $\ker (I-T)$ and define 
$$Q:L_{q,\sigma}(\Sigma,\mathsf{T}\Sigma)\to \ker (I-T)$$ 
by
$$Qu=\sum_{k=1}^m(u|v_k)_\Sigma v_k.$$
Then it can be readily checked that $Q$ is a projection onto $\ker (I-T)$ and it is continuous in $L_{q,\sigma}(\Sigma,\mathsf{T}\Sigma)$ for \emph{any} $q\in (1,\infty)$, since $v_k\in H_q^1(\Sigma,\mathsf{T}\Sigma)$ for any $q\in (1,\infty)$ (using a bootstrap argument). Consequently, the operator 
$$I-T:{\rm ran}(I-Q)\to{\rm ran}(I-T)$$ 
is invertible with bounded inverse.

We use the resolvent representation
$$(I-T)u={S(\lambda)f}+R(\lambda)f,$$
to conclude that the right hand side belongs to ${\rm ran}(I-T)$, hence
$$
(I-Q)u
=(I-T)^{-1}\left({S(\lambda)f}+R(\lambda)f\right).
$$
Let $\phi_0\in (\phi_{A_S}^\infty,\pi/2)$ such that $\phi_{A_{S,k}^G}^\infty\le \phi_0$ for each $k\in\{1,\ldots,N\}$. For $h\in H_0(\Sigma_{\phi})$, $\phi\in (\phi_0,\pi/2)$, we then obtain
\begin{multline*}
(I-Q)h(\omega+A_{S,\Sigma})f
=(I-T)^{-1}\Big(P_{H,\Sigma}\sum_{k=1}^N\psi_k((h(\omega+A_{S,k}^G)P_k^G\bar{f}_{(k)}|e_\ell)\bar{\tau}_{(k)\ell})\circ\varphi_k\\
+\frac{1}{2\pi i}\int_\Gamma h(-\lambda)R(\lambda)fd\lambda\Big),
\end{multline*}
with $\Gamma=(\infty,0]e^{-i(\pi-\theta)}\cup [0,\infty)e^{i(\pi-\theta)}$, $\theta\in (\phi_0,\phi)$.
Estimate \eqref{3/2-decay} then yields
$$|(I-Q)h(\omega+A_{S,\Sigma})|_{\mathcal{B}(L_q(\Sigma))}\le C|h|_\infty,$$
since each of the operators $\omega+A_{S,k}^G$ has a bounded $H^\infty$-calculus.
The remaining part $Qh(\omega+A_{S,\Sigma})$ may be treated as follows.
\begin{align*}
Qh(\omega+A_{S,\Sigma})f&=Q\frac{1}{2\pi i}\int_\Gamma h(-\lambda)(\lambda+\omega)(\lambda+\omega+A_{S,\Sigma})^{-1}f\frac{d\lambda}{\lambda+\omega}\\
&=Q\frac{1}{2\pi i}\int_\Gamma h(-\lambda)(I-A_{S,\Sigma}(\lambda+\omega+A_{S,\Sigma})^{-1})f\frac{d\lambda}{\lambda+\omega}\\
&=Qh(\omega) f-Q\frac{1}{2\pi i}\int_\Gamma h(-\lambda)A_{S,\Sigma}(\lambda+\omega+A_{S,\Sigma})^{-1}f\frac{d\lambda}{\lambda+\omega}.
\end{align*}
For the last integral, we employ the definition of the projection $Q$ from above to obtain
\begin{align*}
Q\frac{1}{2\pi i}&\int_\Gamma h(-\lambda)A_{S,\Sigma}(\lambda+\omega+A_{S,\Sigma})^{-1}f\frac{d\lambda}{\lambda+\omega}\\
&=\frac{1}{2\pi i}\sum_{k=1}^mv_k\int_\Gamma h(-\lambda)(A_{S,\Sigma}(\lambda+\omega+A_{S,\Sigma})^{-1}f|v_k)_\Sigma \frac{d\lambda}{\lambda+\omega}\\
&=2\mu_s\frac{1}{2\pi i}\sum_{k=1}^mv_k\int_\Gamma h(-\lambda)(\mathcal{D}_\Sigma(\lambda+\omega+A_{S,\Sigma})^{-1}f:\nabla_\Sigma v_k) \frac{d\lambda}{\lambda+\omega}.
\end{align*}
By \eqref{eq:sectestimate}, we therefore obtain
$$|Qh(\omega+A_{S,\Sigma})|_{\mathcal{B}(L_q(\Sigma))}\le C|h|_\infty,$$
for some constant $C>0$. Consequently, the operator $\omega+A_{S,\Sigma}$ admits a bounded $H^\infty$-calculus in $L_{q,\sigma}(\Sigma,\mathsf{T}\Sigma)$ with angle $\phi_{A_{S,\Sigma}}^\infty<\pi/2$ provided $\omega>0$ is large enough. An application of \cite[Corollary 3.3.15]{PrSi16} finally yields that it is enough to require $\omega>s(-A_{S,\Sigma})$. This completes the proof of Theorem \ref{thm:Hinftycalc}.

\section{Critical spaces}

\subsection{Strong setting}

We consider the abstract system
\begin{equation}
\label{Abstract-Navier-Stokes}
\partial_t u+A_{S,\Sigma}u=F_\Sigma(u),\quad t>0,\quad u(0)=u_0,
\end{equation}
where $F_\Sigma(u):=-P_{H,\Sigma}\mathcal{P}_\Sigma(u\cdot\nabla_\Sigma u)= -P_{H,\Sigma}\nabla_u u$. Let $q\in (1,\infty)$,
$$X_0:=L_{q,\sigma}(\Sigma,\mathsf{T}\Sigma)\quad\text{and}
\quad X_1:=D(A_{S,\Sigma})=H_{q,\sigma}^2(\Sigma,\mathsf{T}\Sigma).
$$
Furthermore, let $X_\beta=[X_0,X_1]_\beta$ for some $\beta\in (0,1)$, where $[\cdot,\cdot]_\beta$ denotes the complex interpolation functor. Then, by Theorem \ref{thm:Hinftycalc}, it holds that $X_\beta=D((\omega+A_{S,\Sigma})^\beta)$ for $\omega>s(-A_{S,\Sigma})$. In \cite[Section 3.5]{PSW20}, we determined the real and complex interpolation spaces
$(X_0,X_1)_{\alpha,p}$ and $[X_0,X_1]_\alpha$ as
\begin{align*}
[X_0,X_1]_\alpha &=H_{q,\sigma}^{2\alpha}(\Sigma,\mathsf{T}\Sigma), \\
(X_0,X_1)_{\alpha,p}&=B_{qp,\sigma}^{2\alpha}(\Sigma,\mathsf{T}\Sigma),
\end{align*}
for $\alpha\in (0,1)$ and $p\in (1,\infty),$
where
$H_{q,\sigma}^{s}(\Sigma,\mathsf{T}\Sigma):= H_{q}^{s}(\Sigma,\mathsf{T}\Sigma)\cap L_{q,\sigma}(\Sigma,\mathsf{T}\Sigma)$ and
$B_{qp,\sigma}^{s}(\Sigma,\mathsf{T}\Sigma):= B_{qp}^{s}(\Sigma,\mathsf{T}\Sigma)\cap L_{q,\sigma}(\Sigma,\mathsf{T}\Sigma)$
for  $s\in (0,2)$.

\medskip
By H\"older's inequality, the estimate
$$|F_\Sigma(u)|_{L_q(\Sigma)}\le C|u|_{L_{qr'}(\Sigma)}|u|_{H_{qr}^1(\Sigma)}$$
holds. We choose $r,r'\in (1,\infty)$, $1/r+1/r'=1$, in such a way that
$$1-\frac{d}{qr}=-\frac{d}{qr'},\quad\text{or equivalently,}\quad \frac{d}{qr}=\frac{1}{2}\left(1+\frac{d}{q}\right),$$
which is feasible if $q\in (1,d)$. Next, by Sobolev embedding, we have
$$[X_0,X_1]_{\beta}\subset H_{q}^{2\beta}(\Sigma,\mathsf{T}\Sigma)\hookrightarrow H_{qr}^1(\Sigma,\mathsf{T}\Sigma)\cap L_{qr'}(\Sigma,\mathsf{T}\Sigma),$$
provided
$$2\beta-\frac{d}{q}=1-\frac{d}{qr},\quad\text{or equivalently,}\quad
\beta=\frac{1}{4}\left(\frac{d}{q}+1\right).$$
The condition $\beta<1$ requires $q>d/3$, hence $q\in (d/3,d)$. For $q\in (d/3,d)$ we define the \emph{critical weight} by
\begin{equation*}  %\label{mu-c}
\mu_c:=\frac{1}{2}\left(\frac{d}{q}-1\right)+\frac{1}{p},
\end{equation*}
with $2/p+d/q\le 3$, so that $\mu_c\in (1/p,1]$. We consider now the problem
\begin{equation}\label{Abstract-Navier-Stokes-omega}
\partial_t u+\omega u+A_{S,\Sigma}u=F_\Sigma(u)+\omega u,\quad t>0,\quad u(0)=u_0,
\end{equation}
where $\omega>s(-A_{S,\Sigma})$. It is clear that $u$ is a solution of \eqref{Abstract-Navier-Stokes} if and only if $u$ is a solution to \eqref{Abstract-Navier-Stokes-omega}. Note that for each $\omega>s(-A_{S,\Sigma})$, the operator $\omega+A_{S,\Sigma}$ admits a bounded $H^\infty$-calculus in $X_0$ with $H^\infty$-angle $\phi_{A_{S,\Sigma}}^\infty<\pi/2$. We may therefore apply \cite[Theorem 1.2]{PSW18} to \eqref{Abstract-Navier-Stokes-omega} which yields the following result.
\goodbreak
%%%%%%%%%%%%%%%%
\begin{theorem}
\label{thm:strong-solution}
Let $p\in (1,\infty)$, $q\in (d/3,d)$ such that $\frac{2}{p}+\frac{d}{q}\le 3$. {Then} for any initial value $u_0\in B_{qp,\sigma}^{d/q-1}(\Sigma,\mathsf{T}\Sigma)$ there exists a unique {strong} solution
$$u\in H_{p,\mu_c}^1((0,a);L_{q,\sigma}(\Sigma,\mathsf{T}\Sigma))\cap L_{p,\mu_c}((0,a);H_{q,\sigma}^2(\Sigma,\mathsf{T}\Sigma))$$
of \eqref{Abstract-Navier-Stokes} for some $a=a(u_0)>0$, with $\mu_c=1/p+d/(2q)-1/2$. The solution exists on a maximal time interval $[0,t_+(u_0))$ and depends continuously on $u_0$. In addition, we have
$$u\in C([0,t_+);B_{qp,\sigma}^{d/q-1}(\Sigma,\mathsf{T}\Sigma))\cap C((0,t_+);B_{qp,\sigma}^{2-2/p}(\Sigma,\mathsf{T}\Sigma))$$
which means that the solution regularizes instantaneously if $2/p+d/q<3$.
\end{theorem}
\begin{remark}
Note that in case $d=3$ and $p=q=2$, the initial value belongs to 
$$B_{22,\sigma}^{1/2}(\Sigma,\mathsf{T}\Sigma)=H_{2,\sigma}^{1/2}(\Sigma,\mathsf{T}\Sigma).$$
Hence, the celebrated result of Fujita \& Kato~\cite{FuKa62} holds true for the surface Navier-Stokes equations.
\end{remark}
%%%%%%%%%%%%%%%%
\subsection{Weak setting}

In order to cover the case $q\ge d$, we proceed as follows. Let {$A_0=\omega+A_{S,\Sigma}$, $\omega>s(-A_{S,\Sigma})$} and recall that $X_0=L_{q,\sigma}(\Sigma,\mathsf{T}\Sigma)$. By \cite[Theorems V.1.5.1 \& V.1.5.4]{Ama95}, the pair $(X_0,A_0)$ generates an interpolation-extrapolation scale $(X_\alpha,A_\alpha)$, $\alpha\in\mathbb{R}$, with respect to the complex interpolation functor. Note that for $\alpha\in (0,1)$, $A_\alpha$ is the $X_\alpha$-realization of $A_0$ (the restriction of $A_0$ to $X_\alpha$) and
$$X_\alpha = D(A_0^\alpha)=[X_0,X_1]_\alpha=H_{q,\sigma}^{2\alpha}(\Sigma,\mathsf{T}\Sigma), $$
since $A_0$ admits a bounded $H^\infty$-calculus.

Let $X_0^\sharp:=(X_0)'$ and $A_0^\sharp:=(A_0)'$ with $D(A_0^\sharp)=:X_1^\sharp$. Then $(X_0^\sharp,A_0^\sharp)$ generates an interpolation-extrapolation scale $(X_\alpha^\sharp,A_\alpha^\sharp)$, the dual scale, and by \cite[Theorem V.1.5.12]{Ama95}, it holds that
$$(X_\alpha)'=X^\sharp_{-\alpha}\quad\text{and}\quad (A_\alpha)'=A^\sharp_{-\alpha}$$
for $\alpha\in \mathbb{R}$.
Choosing $\alpha=1/2$ in the scale $(X_\alpha,A_\alpha)$, we obtain an operator
$$A_{-1/2}:X_{1/2}\to X_{-1/2},$$
where
$X_{-1/2}=(X_{1/2}^\sharp)'$ (by reflexivity) and, since also $A_0^\sharp$ has a bounded $H^\infty$-calculus,
$$X_{1/2}^\sharp = D((A_0^\sharp)^{1/2})=[X_0^\sharp,X_1^\sharp]_{1/2}
=H_{q',\sigma}^1(\Sigma,\mathsf{T}\Sigma),$$
with $p'=p/(p-1)$ being the conjugate exponent to $p\in (1,\infty)$.  Moreover, we have $A_{-1/2}=(A_{1/2}^\sharp)'$ and $A_{1/2}^\sharp$ is the restriction of $A_0^\sharp$ to $X_{1/2}^\sharp$. Thus, the operator
$A_{-1/2}:X_{1/2}\to X_{-1/2}$
inherits the property of a bounded $H^\infty$-calculus with $H^\infty$-angle $\phi_{A_{-1/2}}^\infty<\pi/2$ from the operator $A_0$.

Since $A_{-1/2}$ is the closure of $A_0$ in $X_{-1/2}$ it follows that $A_{-1/2}u=A_0u$ for $u\in X_1=D(A_0)=H_{q,\sigma}^2(\Sigma,\mathsf{T}\Sigma)$ and thus, for all $v\in X_{1/2}^\sharp$, it holds that
$$\la A_{-1/2}u,v\ra=(A_0u|v)_{\Sigma}=2\mu_s\int_\Sigma \cD_\Sigma(u):\cD_\Sigma(v)\, d\Sigma+{\omega(u|v)_\Sigma},$$
where we made use of the surface divergence theorem. Using that $X_1$ is dense in $X_{1/2}$, we obtain the identity
$$
\la A_{-1/2}u,v\ra=2\mu_s\int_\Sigma \cD_\Sigma(u):\cD_\Sigma(v)\, d\Sigma+{\omega(u|v)_\Sigma},
$$
valid for all $(u,v)\in X_{1/2}\times X_{1/2}^\sharp$. We call the operator $A_{S,\Sigma}^{\sf w}:=A_{-1/2}-\omega$ the \textbf{weak surface Stokes operator}, given by its representation
$$\la A_{S,\Sigma}^{\sf w}u,v\ra=2\mu_s\int_\Sigma \cD_\Sigma(u):\cD_\Sigma(v)\, d\Sigma,\quad (u,v)\in X_{1/2}\times X_{1/2}^\sharp.$$
Multiplying \eqref{Abstract-Navier-Stokes} by a function $\phi\in X_{1/2}^\sharp= H_{q',\sigma}^1(\Sigma,\mathsf{T}\Sigma)$ and using the surface divergence theorem, we obtain the weak formulation
\begin{equation}
\label{eq:weakSurfaceNS}
\partial_t u+A_{S,\Sigma}^{\sf w}u=F_\Sigma^{\sf w}(u),\quad u(0)=u_0,
\end{equation}
in $X_{-1/2}$, where
$$\la F_\Sigma^{\sf w}(u),\phi\ra:=(\nabla_u\phi|u)_\Sigma.$$
Note that $u$ is a solution of \eqref{eq:weakSurfaceNS} if and only if $u$ solves
\begin{equation}
\label{eq:weakSurfaceNS-omega}
\partial_t u+\omega u+A_{S,\Sigma}^{\sf w}u=F_\Sigma^{\sf w}(u)+\omega u,\quad u(0)=u_0,
\end{equation}
in $X_{-1/2}$.
% Since $A_{-1/2}=\omega+A_{S,\Sigma}^{\sf w}$ admits a bounded $H^\infty$-calculus in $X_{-1/2}$ with $H^\infty$-angle $\phi_{A_{-1/2}}^\infty<\pi/2$, 
We will apply Theorem 1.2 in \cite{PSW18} to \eqref{eq:weakSurfaceNS-omega} with the choice
 $$X_0^{\sf w}=X_{-1/2}\quad \text{and}\quad X_1^{\sf w}=X_{1/2}.$$
 For that purpose, we will first characterize some relevant interpolation spaces.
Let
 \begin{equation}
 \label{spaces}
 \begin{aligned}
 &H^s_{q,\sigma}(\Sigma,\mathsf{T}\Sigma):=
 \left\{
 \begin{aligned}
&H^s_q(\Sigma,\mathsf{T}\Sigma)\cap L_{q,\sigma}(\Sigma,\mathsf{T}\Sigma), & \qquad  0\le s \le 1,\\
&\big(H^{-s}_{q',\sigma}(\Sigma,\mathsf{T}\Sigma)\big)',   & \ \ \qquad -1\le s<0,
 \end{aligned}
 \right. \\
 & & \\
 &B^s_{qp,\sigma}(\Sigma,\mathsf{T}\Sigma):=
 \left\{
 \begin{aligned}
&B^s_{qp}(\Sigma,\mathsf{T}\Sigma)\cap L_{q,\sigma}(\Sigma,\mathsf{T}\Sigma), & \quad 0 < s \le 1,\\
&(X_0^{\sf w},X_1^{\sf w})_{1/2,p},& \quad s=0,\\
&\big(B^{-s}_{q'p',\sigma}(\Sigma,\mathsf{T}\Sigma)\big)',   & \quad -1\le s<0.
 \end{aligned}
 \right.
 \end{aligned}
 \end{equation}
By similar arguments as in \cite[Section 2.3]{PrWi18} we obtain
\begin{equation}
\label{interpolation}
 \begin{aligned}
 & [X_{-1/2}, X_{1/2}]_\theta     \ \, =\  H^{2\theta -1}_{q,\sigma}(\Sigma,\mathsf{T}\Sigma),\quad && \theta \in (0,1), \\
 & (X_{-1/2}, X_{1/2})_{\theta,p}  =  B^{2\theta -1}_{qp,\sigma}(\Sigma,\mathsf{T}\Sigma),\quad && \theta \in (0,1) . \\
\end{aligned}
 \end{equation}
Next we show that the nonlinearity
$F_\Sigma^{\sf w}:X_\beta^{\sf w} \to X_{0}^{\sf w} $
is well defined, where
$X_\beta^{\sf w}$ denotes the complex interpolation space, that is, $X^{\sf w}_\beta:=[X_0^{\sf w},X_1^{\sf w}]_{\beta} $ for $\beta\in (1/2,1).$
By \eqref{spaces}, \eqref{interpolation} and Sobolev embedding, we have
\begin{equation}
\label{embedding}
X_\beta^{\sf w} \hookrightarrow
H_q^{2\beta-1}(\Sigma,\mathsf{T}\Sigma) \hookrightarrow L_{2q}(\Sigma,\mathsf{T}\Sigma),
\end{equation}
provided that $2\beta-1\ge \frac{d}{2q}$.
From now on, we assume $2\beta-1=\frac{d}{2q}$, which means $q>d/2$ as $\beta<1$.
Then, by H\"older's inequality and \eqref{embedding}, we obtain
$$|\la F_\Sigma^{\sf w}(u),\phi\ra|\le |u|_{L_{2q}(\Sigma)}^2 |\phi|_{H^1_{q'}(\Sigma)}
\le C |u|^2_{X_\beta^{\sf w}}\,|\phi|_{H^1_{q'}(\Sigma)},$$
showing that
$$F_\Sigma^{\sf w}: X_\beta^{\sf w} \to X^{\sf w}_0\quad\text{with}\quad  |F_\Sigma^{\sf w}(u)|_{X_0^{\sf w}}\le C |u|^2_{X^{\sf w}_\beta}.$$
For $2\beta-1=d/2q$, the critical weight $\mu_c\in (1/p,1]$ is given by $\mu_c:=\mu^{\sf w}_c=1/p+d/2q$ and the corresponding critical trace space in the weak setting reads
\begin{equation*} %\label{weak-trace-space}
X_{\gamma,\mu_c}^{\sf w}=(X_0^{\sf w},X_1^{\sf w})_{\mu_c-1/p,p}= B_{qp,\sigma}^{d/q-1}(\Sigma,\mathsf{T}\Sigma).
\end{equation*}
Note that in case $q\in (d/2,d)$,
the critical spaces in the weak and strong setting coincide.

The existence and uniqueness result for \eqref{eq:weakSurfaceNS} in critical spaces reads as follows.
\goodbreak
%%%%%%%%%%%%%%%
\begin{theorem}
\label{thm:weakSrurfaceStokes}
Let $p\in (1,\infty)$ and $q\in (d/2,\infty)$ such that $\frac{2}{p}+\frac{d}{q}\le 2$.
Then for any initial value  $u_0\in B_{qp,\sigma}^{d/q-1}(\Sigma,\mathsf{T}\Sigma)$ there exists a unique solution
$$u\in H_{p,\mu_c}^1((0,a); H^{-1}_{q,\sigma}(\Sigma,\mathsf{T}\Sigma)
)\cap L_{p,\mu_c}((0,a); H^1_{q,\sigma}(\Sigma,\mathsf{T}\Sigma))$$
of \eqref{eq:weakSurfaceNS} for some $a=a(u_0)>0$, with $\mu_c =1/p+d/2q$. 
The solution exists on a maximal time interval $[0,t_+(u_0))$ and depends continuously on $u_0$. Moreover,
\begin{equation*}
u\in C([0,t_+); B_{qp,\sigma}^{d/q-1}(\Sigma,\mathsf{T}\Sigma)).
\end{equation*}
Suppose  $p>2$ and $q\ge d$. Then each solution with $u_0\in B_{qp,\sigma}^{d/q-1}(\Sigma,\mathsf{T}\Sigma)$ satisfies in addition
\begin{equation*}
u\in H^1_{p,{\rm loc}}((0,t^+); L_{q,\sigma}(\Sigma, {\sf T}\Sigma))\cap L_{p,{\rm loc}}((0, t^+); H^2_{q,\sigma}(\Sigma, {\sf T}\Sigma)).
%\cap C((0, t^+), B^{2-2/p}_{qp,\sigma}(\Sigma, {\sf T}\Sigma)).
\end{equation*}
Hence, in this case, each solution regularizes instantaneously and becomes a strong solution.
\end{theorem}
%%%%%%%%%%%%%
\begin{proof}
 Since $A_{-1/2}=\omega+A_{S,\Sigma}^{\sf w}$ admits a bounded $H^\infty$-calculus in $X_{-1/2}$ with $H^\infty$-angle $\phi_{A_{-1/2}}^\infty<\pi/2$,
 the first assertion follows readily from \cite[Theorem 1.2]{PSW18}.
 
 \smallskip\noindent
 Suppose that $p>2$ and $q\ge d$. Then $\mu_c=1/p + d/2q<1$ and \cite[Theorem 1.2]{PSW18} yields
$u\in  C((0,t_+); B^{1-2/p}_{qp,\sigma} (\Sigma,\mathsf{T}\Sigma)).$
Therefore, 
\begin{equation*}
u(t)\in X_{\gamma,1}^{\sf w}=(X_{-1/2},X_{1/2})_{1-1/p,p}=B_{qp,\sigma}^{1-2/p}(\Sigma,\mathsf{T}\Sigma),
\end{equation*}
for $t\in (0,t_+(u_0))$. 
Noting that for $\mu\in (1/p,1/2]$ we have the embedding
$$B_{qp,\sigma}^{1-2/p}(\Sigma,\mathsf{T}\Sigma)\hookrightarrow B_{qp,\sigma}^{2\mu-2/p}(\Sigma,\mathsf{T}\Sigma)$$
at our disposal, we may now solve \eqref{Abstract-Navier-Stokes} by \cite[Theorem 3.5]{PSW20} with initial value $u(t_0)$, $t_0\in (0,t_+(u_0))$, to obtain a strong solution.
The assertion of the theorem follows now from uniqueness, see  \cite[Theorem 3.4(c)]{MPS19}.
\end{proof}
%%%%%%%%%%
\begin{corollary}
\label{cor:Lq}
Suppose  $d\ge 2$. Then equation~\eqref{eq:weakSurfaceNS} has for each initial value $u_0\in L_{d,\sigma}(\Sigma, {\sf T}\Sigma)$
a unique solution which enjoys the regularity properties of Theorem~\ref{thm:weakSrurfaceStokes}
with $q=d$ for each fixed $p$ with  $p>2$ and $p\ge d$.
\end{corollary}
%%%%%%%%%
\begin{proof}
Suppose $p>2$ and $p\ge d$. Then the 
embedding 
$
L_{d,\sigma}(\Sigma, {\sf T}\Sigma)\hookrightarrow B^0_{dp,\sigma}(\Sigma, {\sf T}\Sigma)
$
holds, see for instance \cite[Theorem 6.2.4 and Theorem 6.4.4]{BeLo76}.
The assertion follows now from Theorem~\ref{thm:weakSrurfaceStokes} with $q=d$.
\end{proof}
\medskip
We now consider the limiting case $p=q=d=2$.
In this case, we have $\mu_c =1$ and the corresponding critical trace space is given by
\begin{equation*} %\label{trace-space-L2}
X_{\gamma,\mu_c}^{\sf w}=(X_{-1/2},X_{1/2})_{1/2,2}=
[X_{-1/2},X_{1/2}]_{1/2}=L_{2,\sigma}(\Sigma,\mathsf{T}\Sigma),
\end{equation*}
see for example \cite[Remark 1.18.10.3]{Tri78}.

We may now also extend each weak solution to a strong solution. To this end, observe that
$$X_{\gamma,\mu_c}^{\sf w}=(X_{-1/2},X_{1/2})_{1/2,2}\hookrightarrow (X_{-1/2},X_{1/2})_{1/2,r}$$
for any $r\in (2,\infty)$. Then we solve \eqref{eq:weakSurfaceNS} by Theorem \ref{thm:weakSrurfaceStokes}
with $u_0\in (X_{-1/2},X_{1/2})_{1/2,r}$ and $\mu_c=1/r+1/2$ (choosing $p=r$ and $q=d=2$). This yields
$$u\in H_{r,\mu_c}^1((0,a);X_0^{\sf w})\cap L_{r,\mu_c}((0,a);X_1^{\sf w})$$
with  $X_1^{\sf w}=H_{2,\sigma}^1(\Sigma,\mathsf{T}\Sigma)$ and $X_0^{\sf w}=H^{-1}_{2,\sigma}(\Sigma,\mathsf{T}\Sigma)$. Since now
$u(t)\in B_{2r,\sigma}^{1-2/r}(\Sigma,\mathsf{T}\Sigma)$ for $t\in (0,a)$, we may argue as above to conclude that the weak solution regularizes to a strong solution
\begin{equation*}
u\in H^1_{r,{\rm loc}}((0,a); L_{2,\sigma}(\Sigma,\mathsf{T}\Sigma))\cap L_{r,{\rm loc}}((0,a); H^2_{2,\sigma}(\Sigma,\mathsf{T}\Sigma)).
\end{equation*}
Moreover,
\begin{equation*}
u(t)\in (L_{2,\sigma}(\Sigma,\mathsf{T}\Sigma), H^2_{2,\sigma}(\Sigma,\mathsf{T}\Sigma))_{1-1/r ,r}=B^{2-2/r}_{2r,\sigma}(\Sigma,\mathsf{T}\Sigma)
\hookrightarrow B^{2\mu-2/p}_{qp,\sigma}(\Sigma,\mathsf{T}\Sigma)
\end{equation*}
for $t\in (0,a)$,
provided $q\ge 2$, $p\ge r$ and
$2-2/r-d/2 \ge 2\mu -2/p\ - d/q$. For $d=2$ this means
$$\mu \le 1/2 + 1/p + 1/q -1/r.$$
We now solve \eqref{Abstract-Navier-Stokes} with initial value
$u(t_0)\in B^{2\mu-2/p}_{qp,\sigma}(\Sigma,\mathsf{T}\Sigma)$ by \cite[Theorem 3.5]{PSW20}
to obtain a solution
$$
v \in H^1_{p,\mu}((0,a); L_{q,\sigma}(\Sigma,\mathsf{T}\Sigma))\cap L_{p,\mu}((0,a); H^2_{q,\sigma}(\Sigma,\mathsf{T}\Sigma)).
$$
As in the proof of \cite[Theorem 3.4(c)]{MPS19} we conclude that uniqueness holds, that is, $v(t)=u(t_0 +t)$.
As $t_0$ can be chosen arbitrarily small, this implies that $u$ shares the regularity properties of $v$ for $ t>0$.
We have shown the following result.
\goodbreak
%%%%%%%%%%
\begin{theorem}
\label{weak-strong-L2}
Let $d=2$.
Then for any $u_0\in L_{2,\sigma}(\Sigma,\mathsf{T}\Sigma)$, problem \eqref{eq:weakSurfaceNS} admits a unique solution
\begin{equation*}
u\in H^1_{2}((0, a), H^{-1}_{2,\sigma}(\Sigma,\mathsf{T}\Sigma))\cap L_{2}((0, a);H^1_{2,\sigma}(\Sigma,\mathsf{T}\Sigma))
\end{equation*}
for some $a=a(u_0)>0$. The solution exists on a maximal time interval $[0,t^+(u_0))$.
In addition, we have
\begin{equation}
\label{continuous-H1}
u\in  C([0,t^+); L_{2,\sigma}(\Sigma,\mathsf{T}\Sigma)).
\end{equation}

\noindent
Furthermore, each solution satisfies
\begin{equation*}
u\in H^1_{p,{\rm loc}}((0, t^+); L_{q,\sigma}(\Sigma,\mathsf{T}\Sigma))\cap L_{p,{\rm loc}}((0, t^+) ; H^2_{q,\sigma}(\Sigma,\mathsf{T}\Sigma))
\end{equation*}
for any fixed $q\in (1,\infty)$ and any fixed $p\in (2,\infty)$.
Therefore, any solution with initial value $u_0\in L_{2,\sigma}(\Sigma,\mathsf{T}\Sigma)$ regularizes instantaneously and becomes a strong $L_p$-$L_q$ solution.
\end{theorem}
%%%%%%%%%%%
\begin{proof}
According to the considerations preceding the Theorem, the assertions hold true for a fixed $q\ge 2$ and $p\in (2, \infty).$
The case $q\in (1,2)$ follows readily from the result in the $L_2$-case and the embedding
$L_2(\Sigma,\mathsf{T}\Sigma)\hookrightarrow L_q(\Sigma,\mathsf{T}\Sigma)$.
\end{proof}

%%%%%%%%%%%%%
\subsection{Energy estimates and global existence}\label{subsec:EnergyEst}

In \cite{PSW20} we showed that the set of equilibria $\mathcal E$ for~\eqref{NS-surface-introduction}, 
respectively~\eqref{Abstract-Navier-Stokes},
 consists exactly of the Killing vector fields on $\Sigma$, that is,
\begin{equation*} %\label{equilibria}
\mathcal E=\{u\in C^\infty(\Sigma,{\sf T}\Sigma)\mid \mathcal D_\Sigma(u)=0\}.
\end{equation*}
We recall that the condition $\mathcal D_\Sigma(u)=0$ implies that $u$ is divergence free (which follows from the
relation ${\rm div}_\Sigma\,u = {\sf tr}\,\mathcal D_\Sigma (u)$).
Moreover, one can show that any vector field $u\in H^1_q(\Sigma,{\sf T}\Sigma)$ satisfying $\mathcal D_\Sigma(u)=0$
is already smooth,  see for instance \cite[Lemma 3]{Pri94}.
Lastly, we recall that $\mathcal E$ is a finite dimensional vector space. If fact,
${\rm dim}\,\mathcal E\le d(d+1)/2$, with equal sign for the case where $\Sigma$ is isometric to a Euclidean sphere,
see for instance the remarks in Section 4.1 of~\cite{PSW20}.

Let us define the space
\begin{equation}
\label{V2}
V^j_2(\Sigma):=\{v\in H^j_{2,\sigma}(\Sigma,\mathsf{T}\Sigma) \mid (v|z)_\Sigma=0\ \text{for all}\  z\in\mathcal E\},\quad j\in\{0,1\}.
\end{equation}
Note that $V^j_2(\Sigma)$ is a closed subspace of $ H^j_{2,\sigma}(\Sigma,\mathsf{T}\Sigma)$, and hence is a Banach space.
Moreover, $H^j_{2,\sigma}(\Sigma, {\sf T}\Sigma)=\mathcal E \oplus V^j_2(\Sigma)$, see Remark \ref{rem:convergence}(a).

From now on we assume that $d=2$, and
we show that any solution of \eqref{eq:weakSurfaceNS} with initial value $v_0\in L_{2,\sigma}(\Sigma)$ being  orthogonal to  $\mathcal E$ will remain orthogonal
for all later times. Moreover, we establish an energy estimate for such solutions.
%%%%%%%%%%%%%
\begin{proposition}
\label{pro:solutions-orthogonal}
Let $d=2$. Suppose $v_0\in V^0_2(\Sigma)$ and
let $v$ be the solution of \eqref{eq:weakSurfaceNS} established in Theorem~\ref{weak-strong-L2}.
Then
\begin{itemize}
\vspace{2mm}
\item[(a)]
 $v(t)\in V^0_2(\Sigma)$ for $t\in [0,t^+(v_0))$ and $v(t)\in V^1_2(\Sigma)$  for $t\in (0, t^+(v_0)).$
\vspace{2mm}
\item[(b)] There exists a universal constant $M>0$ such that
\begin{equation*} %\label{estimate}
|v(t)|^2_{L_2(\Sigma)}+ \int_0^t |v(s)|^2_{H^1_2(\Sigma)}\, ds \le M |v_0|^2_{L_2(\Sigma)}, \quad t\in (0,t^+(v_0)).
\end{equation*}
\end{itemize}
\end{proposition}
%%%%%%%%%%%%%
\begin{proof}
(a)
According to Theorem~\ref{weak-strong-L2}, we know that
$$
v\in H^1_{p,{\rm loc}}((0, t^+); L_{2,\sigma}(\Sigma,\mathsf{T}\Sigma))\cap L_{p,{\rm loc}}((0, t^+) ; H^2_{2,\sigma}(\Sigma,\mathsf{T}\Sigma))
$$
for $p>2$. Pick any $z\in \mathcal E$. Then
\begin{equation*}
\frac{d}{dt} (v(t) | z)_\Sigma
=- (\nabla_{v(t)}v(t) | z)_\Sigma - (A_{S,\Sigma} v(t) | z )_\Sigma
=  -(\nabla_{v(t)}v(t) | z)_\Sigma,
\end{equation*}
where the time derivative exists for almost all $t\in (0, t^+(v_0))$.
For the last equal sign we employed the property that $A_{S,\Sigma}$ is symmetric on $L_{2,\sigma}(\Sigma)$ and $N(A_{S,\Sigma})=\mathcal E$,
see \cite[Proposition 4.1]{PSW20}.
In a next step we show that
\begin{equation*}
(\nabla_v v | z )_\Sigma =0\quad\text{for all}\quad  v\in H^1_{2,\sigma}(\Sigma).
\end{equation*}
Indeed, this follows from
\begin{equation*}
\begin{aligned}
 (\nabla_v v | z)_\Sigma &= \int_\Sigma (\nabla_v v | z)\,d\Sigma =   \int_\Sigma \big(\nabla_v  (v | z)- (v| \nabla_v z)\big)\,d\Sigma \\
 &= \int_\Sigma \Big( (v| {\rm grad}(v | z))- (v| \nabla_v z)\Big)\,d\Sigma
    = 0,
 \end{aligned}
\end{equation*}
where we used \eqref{metric-connection}, the surface divergence theorem and the property that $z$ is a Killing vector field,
(which implies $(\nabla_v z| v) + (v | \nabla_v z)=0)$.
Hence, we have shown that
 $\frac{d}{dt} ( v(t)|z )_\Sigma=0$ for almost all $t\in (0, t^+(v_0))$.
\eqref{continuous-H1} now implies
$$(v(t)| z)_\Sigma =(v_0 | z)_\Sigma =0\quad \text{for all}\quad t\in [0,t^+(v_0)).$$
(b) Similarly as in part (a), one shows (suppressing the variable $t$) that
\begin{equation*}
\begin{aligned}
\frac{d}{dt} \frac{1}{2} |v(t)|^2_{L_2(\Sigma)}
&=-\int_{\Sigma} \big( (\nabla_v v|v) + (A_{S,\Sigma} v |v)\big)\,d\Sigma
  =   -2\mu_s\int_{\Sigma} |\mathcal D_\Sigma (v)|^2\,d\Sigma.
\end{aligned}
\end{equation*}
The assertion in part (a) and Korn's inequality~\eqref{Korn} readily imply
\begin{equation}
\label{energy-inequality}
\frac{d}{dt}  |v(t)|^2_{L_2(\Sigma)} + \alpha\, |v(t)|^2_{H^1_2(\Sigma)}\le 0,\quad t\in (0, t^+(v_0)),
\end{equation}
with an appropriate constant $\alpha>0$.  % $\alpha = 4\mu_s/C$
Integration yields the assertion in (b), as $|v|^2_{L_2(\Sigma)}$ is absolutely continuous on $[0,t^+(v_0))$.
\end{proof}

%%%%%%%%%%%
\begin{proposition}
\label{pro:global-existence-v}
Suppose that $d=2$ and  $v_0\in V^0_2(\Sigma)$.

\smallskip
Then  problem \eqref{eq:weakSurfaceNS} admits a unique global solution $v$ enjoying the regularity properties stated in Theorem~\eqref{weak-strong-L2},
with $t^+(v_0)=\infty$.

Moreover, there exists a constant $\alpha>0$ such that
\begin{equation}
\label{exponential-decay}
|v(t)|_{L_2(\Sigma)} \le e^{-\alpha t} |v_0|_{L_2(\Sigma)},  \quad  t\ge 0.
\end{equation}
\end{proposition}

%%%%%%%%%%%%%%%%%%

\bigskip
\begin{proof}
By the abstract result \cite[Theorem 2.4]{PSW18} on global-in-time existence, the maximal time of existence $t_+(u_0)$ satisfies the following property:
$$t_+(u_0)<\infty \Longrightarrow u\notin L_p((0,t_+);[X_0^{\sf w},X_1^{\sf w}]_{\mu_c}).$$
Observe that in case $p=q=d=2$, it holds $\mu_c=1/p+d/(2q)=1$, hence if
$$u\in L_2((0,t_+);H_2^1(\Sigma,\mathsf{T}\Sigma)),$$
then the weak solution exists globally in time.
Proposition~\ref{pro:solutions-orthogonal} guarantees that any solution $v$ with initial value $v_0\in V^0_2(\Sigma)$
satisfies
\begin{equation*}
v\in L_2((0,t^+(v_0)), H^1_2(\Sigma,{\sf T}\Sigma))
\end{equation*}
and, hence, global existence of the weak solution follows.
Since we know that the weak solution in this case regularizes to a strong solution, we obtain global in time existence of strong solutions for $q=d=2$ as well.

\medskip
Finally, we conclude from~\eqref{energy-inequality} that
$ \frac{d}{dt}  |v(t)|^2_{L_2(\Sigma)} + \alpha\, |v(t)|^2_{L_2(\Sigma)}\le 0$ for any $t>0$ 
and this implies the estimate in~\eqref{exponential-decay}.

\end{proof}
%%%%%%%%%%%%%%%%~\eqref{weak-strong-L2}
\begin{remark}
\label{pro:v-linearized}
Suppose $u_*\in \mathcal E$ and $v_0\in V^0_2(\Sigma)$. Then
 the assertions of Theorem~\ref{thm:strong-solution}, Theorem~\ref{weak-strong-L2} as well as Propositions~\ref{pro:solutions-orthogonal} 
 and~\ref{pro:global-existence-v} hold true for solutions of 
  \begin{equation}
  \label{v-linearized-strong}
   \partial_t v + A_{S,\Sigma}v = -P_{H,\Sigma}(\nabla_v v + \nabla_{u_*}v + \nabla_v u_*),\quad v(0)=v_0,
 \end{equation}
 respectively its weak formulation
\begin{equation}
 \label{v-linearized-weak}
   \partial_t v + A^{\sf w}_{S,\Sigma}v = F^{\sf w}_\Sigma(v),  \quad v(0)=v_0,
 \end{equation}
 where $\la F^{\sf w}_\Sigma (v),\phi \ra = ( v | \nabla_v \phi)_\Sigma + (v| \nabla_{u_*}\phi)_\Sigma + (u_* | \nabla_v \phi)_\Sigma $ 
 for $\phi\in H^1_{2,\sigma}(\Sigma,{\sf T}\Sigma)$. 
 
 \medskip
 In particular, each solution of\eqref{v-linearized-weak} with initial value $v_0\in V^0_2(\Sigma)$ exists globally and there
 exists a  positive constant $\alpha$ such that
\begin{equation*} %\label{exponential-decay-2}
|v(t)|_{L_2(\Sigma)} \le e^{-\alpha t} |v_0|_{L_2(\Sigma)},  \quad  t\ge 0.
\end{equation*}
\end{remark}
%%%%%%%%%%%%%%%
\begin{proof}
One readily verifies that the assertions of Theorems~\ref{thm:strong-solution} and~\ref{weak-strong-L2} remain valid for
problem~\eqref{v-linearized-strong} and~\eqref{v-linearized-weak}, respectively.
In fact, one only needs to verify that the terms on the right hand side can be estimated in the same way
as in the proof of Theorems~\ref{thm:strong-solution} and~\ref{weak-strong-L2}.

Next we show that $v(t)\in V^0_2(\Sigma)$ for $t\in [0,t^+(v_0))$.
Let $z\in \mathcal E$. Following the proof of Proposition~\ref{pro:solutions-orthogonal}(a), we obtain
\begin{equation*}
\frac{d}{dt} (v(t) | z)_\Sigma
%=- (\nabla_{v(t)}v(t) | z)_\Sigma - (A_{S,\Sigma} v(t) | z )_\Sigma
=  -(\nabla_{v(t)}v(t)+\nabla_{u_*}v(t) + \nabla_{v(t)}u_* | z)_\Sigma, \quad t\in (0,t^+(v_0)).
\end{equation*}
According to the proof of Proposition~\ref{pro:solutions-orthogonal}(a), $(\nabla_v v|z)=0$ 
and it remains to show that
$(\nabla_{u_*}v+ \nabla_{v}u_* | z)_\Sigma=0$ for any $v\in H^1_{2,\sigma}(\Sigma, {\sf T}\Sigma).$
This follows from 
\begin{equation*}
\begin{aligned}
 (\nabla_{u_*} v + \nabla_v u_*| z)_\Sigma &=  \int_\Sigma \big(\nabla_{u_*} (v | z)+\nabla_v(u_*|z)- (u_*| \nabla_v z)-(v| \nabla_{u_*} z)\big)=0
 %&= \int_\Sigma \Big( (v| {\rm grad}(v | z))- (v| \nabla_v z)\Big)\,d\Sigma
 \end{aligned}
\end{equation*}
where we used \eqref{metric-connection}, the surface divergence theorem and the property that $z$ is a Killing vector field.
The same arguments as in the proof of Propositions~\ref{pro:solutions-orthogonal} and~\ref{pro:global-existence-v}
yield the remaining assertions. 
\end{proof}
%%%%%%%%%%%%%%%%
\begin{theorem}[Global existence]
\label{thm:convergence}
Suppose $d=2$.

Then any solution of \eqref{eq:weakSurfaceNS} with initial value $u_0\in L_{2,\sigma}(\Sigma,{\sf T}\Sigma)$ exists globally, has
the regularity properties listed in Theorem~\ref{weak-strong-L2},
 and converges at an exponential rate
to the equilibrium 
$u_* = P_{\cE} u_0$ in the topology of  $H_q^{2}(\Sigma,\mathsf{T}\Sigma)$ for any fixed  $q\in  (1,\infty)$,
where $P_{\cE}$ is the orthogonal projection of $u_0$ onto $\cE$ with respect to the $L_2(\Sigma, \mathsf{T}\Sigma)$ inner product.
\end{theorem}
%%%%%%%%%%%%%%%%%%%%%%%
\begin{proof}
Let  $u_0\in L_{2,\sigma}(\Sigma,{\sf T}\Sigma)$ be given. Then there exist  unique elements
 $u_*\in\mathcal E$ and  $v_0\in V^0_2(\Sigma)$ such that
 $u_0=u_*+v_0 $.
Let $v$ be the unique (global) solution of problem~~\eqref{v-linearized-weak}, respectively~\eqref{v-linearized-strong},
 whose existence has been asserted in Remark~\ref{pro:v-linearized}
Then
$$u(t):=u_* + v(t),\quad t>0,$$
yields a (unique) global solution of \eqref{Abstract-Navier-Stokes}, respectively \eqref{eq:weakSurfaceNS}, with initial value $u_0$.
For this, we just need to observe that
\begin{equation*}
\partial_t u_* + A_{S,\Sigma} u_* =  - P_{H,\Sigma} \nabla_{u_*} u_*.
\end{equation*}
Indeed, this follows from  the relations $N(A_{S,\Sigma})=\mathcal E$ and $\nabla_{u_*}u_*=\frac{1}{2} {\rm grad}(u_* | u_*)$, with the latter assertion
implying $P_{H,\Sigma} \nabla_{u_*} u_*=0$.
It follows readily that $u(t)\to u_*$ as $t\to\infty$ at an exponential rate in $L_2(\Sigma,\mathsf{T}\Sigma)$. 
To prove convergence in the stronger topology $H_q^{2}(\Sigma,\mathsf{T}\Sigma)$, we proceed as follows. 

\medskip\noindent
{\bf (i)}
First, we note that $L_2(\Sigma,\mathsf{T}\Sigma)\hookrightarrow B_{2r}^0(\Sigma,\mathsf{T}\Sigma)$ for any fixed $r\ge 2$. 
For fixed, but arbitrary $t_1>0$, we solve \eqref{eq:weakSurfaceNS} with initial value $u_1:=u(t_1)\in B_{2r,\sigma}^0(\Sigma,\mathsf{T}\Sigma)$. 
Choosing $\mu_c=1/r+1/2$,
it follows from \cite[Theorem 1.2]{PSW18} that there exist positive numbers $\tau=\tau(u_*)$, $\varepsilon=\varepsilon(u_*)$ and $C_1=C_1(u_*)$ such that
\begin{equation}
\label{E1}
|u(\cdot, \hat u_1) - u(\cdot, u_*)|_{\EE_{1,\mu_c}(0,2\tau)}\le C_1 |\hat u_1-u_*|_{B_{2r,\sigma}^0(\Sigma)}
\end{equation}
for all $\hat u_1\in B_{2r,\sigma}^0(\Sigma,\mathsf{T}\Sigma)$ with $|\hat u_1-u_*|_{B_{2r,\sigma}^0(\Sigma)}<\varepsilon,$
%$\hat u_1\in B_{B_{2r,\sigma}^0(\Sigma,\mathsf{T}\Sigma)}(u_*,\varepsilon)$,
where  
\begin{equation*}
\EE_{1,\mu_c}(0,2\tau) 
=H_{r,\mu_c}^1((0,2\tau);H_{2,\sigma}^{-1}(\Sigma,\mathsf{T}\Sigma))\cap L_{r,\mu_c}((0,2\tau);H_{2,\sigma}^{1}(\Sigma,\mathsf{T}\Sigma)).
\end{equation*}
It should be observed here that $u(t, u_*)=u_*$, as $u_*$ is an equilibrium.

\smallskip
\noindent
By following the arguments in \cite[page 228]{PrSi16}
and employing 
$$X^{\sf w}_\gamma= (X^{\sf w}_0, X^{\sf w}_1)_{1-1/r,r}= (H^{-1}_{2,\sigma}(\Sigma,{\sf T}\Sigma), H^{1}_{2,\sigma}(\Sigma,{\sf T}\Sigma))_{1-1/r, r}=B^{1-2/r}_{2r,\sigma}(\Sigma, {\sf T}\Sigma),$$
see~\eqref{interpolation}, 
we conclude with~\eqref{E1} that there is a constant $C_2=C_2(u_*)$ such that
\begin{equation*} %\label{E2}
|u(\cdot, \hat u_1) - u(\cdot, u_*)|_{C([\tau,2\tau], B^{1-2/r}_{2r,\sigma}(\Sigma))}\le  C_2 |\hat u_1 -u_*|_{B_{2r}^0(\Sigma)}.
\end{equation*}
Letting $C_3$ be the embedding constant of $L_2(\Sigma,{\sf T}\Sigma)\hookrightarrow B^0_{2r}(\Sigma,{\sf T}\Sigma)$ 
we choose $t_1$ large enough such that $|u(t_1)-u_*|_{L_2(\Sigma)}<\varepsilon/C_3$. 
Setting $t_2=\tau$, we may deduce from the estimates above that 
\begin{equation}
\label{eq:conv1}
|u(t_2,u_1)-u_*|_{B_{2r}^{1-2/r}(\Sigma)}\le C|u_1-u_*|_{L_{2}(\Sigma)},
\end{equation}
with $C=C(u_*)=C_2C_3 $.
Here we have also used the uniqueness of solutions, cf. \cite[Theorem 3.4(c)]{MPS19}.

\medskip
\noindent
{\bf (ii)}
For fixed, but arbitrary $r>2$, we now choose a weight $\mu\in (\frac{1}{r},\frac{1}{2}]$ so that
$$B_{2r}^{1-2/r}(\Sigma,\mathsf{T}\Sigma)\hookrightarrow B_{2r}^{2\mu-2/r}(\Sigma,\mathsf{T}\Sigma).$$
Solving \eqref{Abstract-Navier-Stokes} with initial value $u_2:=u(t_2,u_1)\in B_{2r,\sigma}^{2\mu-2/r}(\Sigma,\mathsf{T}\Sigma)$ and repeating
the above procedure in the `strong' spaces $(X_0,X_1)$, we obtain the estimate
\begin{equation}
\label{eq:conv2}
|u(t_3,u_2)-u_*|_{B_{2r}^{2-2/r}(\Sigma)}\le C|u_2-u_*|_{B_{2r}^{1-2/r}(\Sigma)},
\end{equation}
for some $t_3=t_3(u_*)>0$.

\medskip\noindent
{\bf (iii)} Next, we use the Sobolev embedding 
$$B_{2r}^{2-2/r}(\Sigma,\mathsf{T}\Sigma)\hookrightarrow B_{sr}^{2\mu-2/r}(\Sigma,\mathsf{T}\Sigma),$$
valid for $s\ge 2$ and $\mu\in (\frac{1}{r},\frac{1}{2}+\frac{1}{s}]$. 
Choosing $t_1>0$ from above sufficiently large, we infer from~\eqref{eq:conv2} that $u(t_3,u_2)$ is close to $u_*$ in the topology of $B_{sr}^{2\mu-2/r}(\Sigma,\mathsf{T}\Sigma)$.
Solving \eqref{Abstract-Navier-Stokes} with initial value $u_3:=u(t_3,u_2)\in B_{sr,\sigma}^{2\mu-2/r}(\Sigma,\mathsf{T}\Sigma)$ and repeating
the above procedure  in the `strong' spaces $(X_0,X_1)$, we obtain the estimate
\begin{equation}
\label{eq:conv3}
|u(t_4,u_3)-u_*|_{B_{sr}^{2-2/r}(\Sigma)}\le C|u_3-u_*|_{B_{2r}^{2-2/r}(\Sigma)},
\end{equation}
for some $t_4=t_4(u_*)>0$.

\medskip\noindent
{\bf (iv)} We will now consider \eqref{Abstract-Navier-Stokes} in the spaces
$$(X_{1/2}, X_{1+1/2})= (H^1_{s,\sigma}(\Sigma, {\sf T} \Sigma), H^3_{s,\sigma}(\Sigma, {\sf T} \Sigma)),$$
where $(X_\alpha, A_\alpha)$ is the interpolation-extrapolation scale with respect to the complex 
interpolation functor, based
on $X_0=L_{s,\sigma}(\Sigma, {\sf T}\Sigma)$, $s\in (1,\infty)$, introduced at the beginning of Section 4.2.
We note that $A_{1/2}$, the realization of $A_0=\omega + A_{S,\Sigma}$ in $X_{1/2}$, has exactly the same properties as $A_0$.
 For $\beta=1/2$ we obtain 
\begin{equation*}
[X_{1/2}, X_{1+1/2}]_{\beta}=[H^1_{s,\sigma}(\Sigma, {\sf T} \Sigma), H^3_{s,\sigma}(\Sigma, {\sf T} \Sigma)]_{1/2}=H^2_{s,\sigma}(\Sigma, {\sf T} \Sigma).
\end{equation*}
It is easy to see that the nonlinearity 
$$
F_\Sigma: H^2_{s,\sigma}(\Sigma, {\sf T}\Sigma)\times H^2_{s,\sigma}(\Sigma, {\sf T}\Sigma) \to H^1_{s,\sigma}(\Sigma, {\sf T}\Sigma)$$
is bilinear and bounded. We can employ \cite[Theorem 1.2 or 2.1]{PSW18}  for problem \eqref{Abstract-Navier-Stokes} 
to obtain a solution $u(\cdot, u_4)$ with initial value 
$u_4=u(t_4,u_3)\in B_{sr}^{2-2/r}(\Sigma,{\sf T}\Sigma).$ 
To do so, we choose $\mu=1/2$ and verify that 
\begin{equation*}
\begin{aligned}
X_{\gamma,\mu}:&=(X_{1/2}, X_{1+1/2})_{1/2-1/r,r}=B^{2-2/r}_{sr,\sigma}(\Sigma, {\sf T}\Sigma), \\
X_{\gamma,1}:&=(X_{1/2}, X_{1+1/2})_{1-1/r,r}=B^{3-2/r}_{sr,\sigma}(\Sigma, {\sf T}\Sigma).
\end{aligned}
\end{equation*}
Noting that the numbers $\mu=1/2$ and $\beta=1/2$ satisfy the assumptions of \cite{PSW18} Theorem 1.2 or 2.1, we can once more
repeat the procedure outlined in step (i) to obtain
\begin{equation}
\label{eq:conv4}
|u(t_5,u_4)-u_*|_{B_{sr}^{3-2/r}(\Sigma)}\le C|u_4-u_*|_{B_{sr}^{2-2/r}(\Sigma)},
\end{equation}
for some $t_5=t_5(u_*)>0$,  provided $t_1$ is chosen sufficiently large.
Combining the estimates~\eqref{eq:conv1}-\eqref{eq:conv4} and 
using the semiflow property, we obtain the estimate
$$|u(t_5+t_4+ t_3+t_2+t_1,u_0)-u_*|_{B_{sr}^{3-2/r}(\Sigma)}\le C|u(t_1,u_0)-u_*|_{L_2(\Sigma)}.$$
Since $|u(t_1,u_0)-u_*|_{L_2(\Sigma)}\to 0$ at an exponential rate as $t_1\to\infty$, we conclude that
$u(t,u_0)$ converges to $u_*$ at an exponential rate in the topology of $B^{3-2/r}_{sr,\sigma}(\Sigma, {\sf T} \Sigma)$ 
as well, as $t\to \infty$.

\medskip\noindent
 {\bf (iv)}
Finally, given $q\in (1,\infty)$, the embedding 
$B_{sr,\sigma}^{3-2/r}(\Sigma,\mathsf{T}\Sigma)\hookrightarrow H_{q,\sigma}^{2}(\Sigma,\mathsf{T}\Sigma),$
being valid for sufficiently large parameters $s$ and $r$, yields the last assertion of the theorem.
\end{proof}
%%%%%%%%%%%%%
\begin{remarks}
\label{rem:convergence}
\vspace{1mm}
(a)
Let $d\ge 2$ and let ${\mathfrak F}^s(\Sigma)$ denote any of the spaces 
$$H^s_{q,\sigma}(\Sigma, {\sf T}\Sigma),\ B^s_{qp,\sigma}(\Sigma,{\sf T}\Sigma),
\quad\text{where $s\in\R,\ 1<p,q<\infty$.}
$$
Then we have
\begin{equation}
\label{Fs-decomposition}
{\mathfrak F}^s(\Sigma)
= \cE \oplus \{v\in {\mathfrak F}^s(\Sigma): \la v, z\ra_\Sigma =0, \ z\in\cE\}.
\end{equation}
Here, $\la v, z \ra_\Sigma: = (v|z)_\Sigma$ if $s>0$. In case $s\le 0$, 
$$
\la \cdot, \cdot\ra_\Sigma : {\mathfrak F}^{s} (\Sigma) \times ({\mathfrak F}^\prime)^{-s}(\Sigma) \rightarrow \R
$$
denotes the duality pairing, induced by $(\cdot | \cdot)_\Sigma$, where 
$$({\mathfrak F}^\prime)^s(\Sigma)\in \{H^s_{q',\sigma}(\Sigma,{\sf T}\Sigma), B^s_{q'p',\sigma}(\Sigma,{\sf T}\Sigma)\}.$$

\medskip\noindent
Since we can identify  $\cE\subset C^\infty(\Sigma, {\sf T}\Sigma)$  as a subspace of ${\mathfrak F}^s(\Sigma)$,
the expression $\la v, z \ra_\Sigma$ is defined for every $(v,z)\in {\mathfrak F}^s(\Sigma)\times \cE$ and
\eqref{Fs-decomposition} is, therefore, meaningful.

\begin{proof}
We will provide a proof of \eqref{Fs-decomposition}.
As $\cE $ is finite dimensional, we can find a basis $\{z_1,\ldots, z_m\}$ for $\cE$ which has the property
that $(z_i | z_j)_\Sigma =\delta_{ij}$.
With this at hand, we define the projection
\begin{equation}
\label{projection}
P^s_\cE : {\mathfrak F}^s(\Sigma) \to \cE, \quad P^s_\cE v:= \sum_{j=1}^m \la v , z_j \ra z_j,
\end{equation}
onto $\cE$.
This yields the direct topological decomposition  
${\mathfrak F}^s(\Sigma)= \cE \oplus V^s(\Sigma)$,
where 
$ V^s(\Sigma)= (I-P^s_\cE) {\mathfrak F}^s(\Sigma).$
In order to justify \eqref{Fs-decomposition}, it suffices to show that
\begin{equation*}
 V^s(\Sigma) =\hat V^s(\Sigma):=\{v\in {\mathfrak F}^s(\Sigma) \mid \la v, z\ra_\Sigma =0, \ z\in\cE\}.
\end{equation*}
Suppose  $v\in  V^s(\Sigma)$. 
Then $v=(I-P^s_\cE)v$ and we obtain 
\begin{equation*}
\la v, z_j \ra_\Sigma = \la (I-P^s_\cE)v, z_j \ra_\Sigma= \la v, z_j \ra_\Sigma  -\la P^s_\cE v, z_j \ra_\Sigma =0,\quad j=1,\ldots, m,
\end{equation*}
showing that $v\in \hat V^s(\Sigma).$
Suppose now that $v \in \hat V^s(\Sigma)$.
Since ${\mathfrak F}^s(\Sigma)= \cE \oplus V^s(\Sigma)$, 
there are unique elements $(z,w)\in \cE\times V^s(\Sigma)$ such that $v=z+w$.
Then, by the first step,
\begin{equation*}
0=\la v, z \ra_\Sigma = \la z+w, z \ra_\Sigma =\la z,z \ra_\Sigma = |z|^2_{L_2(\Sigma)},
\end{equation*}
hence $z=0$ and therefore $v=w\in V^s(\Sigma)$.
\end{proof}
\noindent
(b) 
It is interesting to note that 
the assertions of Propositions~\ref{pro:solutions-orthogonal} and~\ref{pro:global-existence-v}  
remain valid in case $d>2$, with the following modifications:

\smallskip\noindent
Let $p>2$ and $q\ge d > 2$. Suppose  $v_0\in B^{d/q-1}_{qp,\sigma}(\Sigma, {\sf T}\Sigma) $ satisfies  $\la v_0, z\ra _\Sigma=0$ for every $z\in \mathcal E$,
where $\la v_0, z\ra_\Sigma$ has the same meaning as in (a).
Let $v$ be the unique solution of \eqref{eq:weakSurfaceNS} with initial value $v_0$.

\smallskip
Then there exists a constant $\alpha>0$ such that
\begin{equation}
\label{estimate-d3}
|v(t)|_{L_2(\Sigma)}\le e^{-\alpha (t-\tau)} |v(\tau)|_{L_2(\Sigma)},\quad t\in [\tau ,t^+(v_0)),
\end{equation}
for  fixed $\tau\in [0,t^+(v_0))$, where  $\tau\in (0 ,t^+(v_0))$ in case $v_0\notin L_{2,\sigma}(\Sigma, {\sf T}\Sigma)$.
%%%%%%%%%
\clearpage
\begin{proof}
According to Theorem \ref{thm:weakSrurfaceStokes}, there exists a unique solution $v$ to problem \eqref{eq:weakSurfaceNS} with regularity
\begin{equation}
\label{regularity-combined}
\begin{aligned}
v&\in H_{p,\mu_c}^1((0,a); H^{-1}_{q,\sigma}(\Sigma,\mathsf{T}\Sigma))\cap L_{p,\mu_c}((0,a); H^1_{q,\sigma}(\Sigma,\mathsf{T}\Sigma)) \\
&\cap H^1_{p,{\rm loc}}((0,t^+); L_{q,\sigma}(\Sigma, {\sf T}\Sigma))\cap L_{p,{\rm loc}}((0, t^+); H^2_{q,\sigma}(\Sigma, {\sf T}\Sigma))
\end{aligned}
\end{equation}
for each fixed $a\in (0, t^+)$. For $z\in \cE$ we obtain, as in the proof of Theorem~\ref{pro:solutions-orthogonal},
\begin{equation*}
\frac{d}{dt} \la v(t) , z\ra_\Sigma
=- (\nabla_{v(t)}v(t)| z)_\Sigma - (A_{S,\Sigma} v(t) | z )_\Sigma
=  -(\nabla_{v(t)}v(t) | z)_\Sigma =0,
\end{equation*}
for all $t\in (0 ,t^+(v_0))$.
Therefore, $\la v(\cdot), z\ra_\Sigma \in H^1_{p,\mu_c}(0, a)$ and $\frac{d}{dt} \la v(t) | z\ra_\Sigma=0$ for $t\in (0,t^+(v_0))$.
By \cite[Lemma 2.1(b)]{PrSi04}, or \cite[Lemma 3.2.5(b)]{PrSi16}, we infer that $\la v(\cdot), z\ra_\Sigma \in H^1_{1,{\rm loc}}([0, a])$ for any fixed $a\in (0, t^+(v_0))$.
We can now conclude from the fundamental theorem that 
\begin{equation}
\label{senkrecht}
\la v(t), z\ra_\Sigma =0,\quad t\in [0, t^+(v_0)).
\end{equation}

\smallskip

Let $\tau\in (0,t^+(v_0))$ be fixed.
As $p,q>2$, we conclude from~\eqref{regularity-combined} and \eqref{senkrecht} that
$v(t)\in V^1_2(\Sigma)=\{u\in H^1_2(\Sigma, {\sf T}\Sigma)\mid (u |z)_\Sigma =0,\ z\in \cE\}$
for any $t\in [\tau,t^+(v_0))$. 
Hence, Korn's inequality~\eqref{Korn} holds true for $v(t)$ with $t\in [\tau, t^+(v_0))$,
and we can now follow the proof of Propositions~\ref{pro:solutions-orthogonal}(b) and \ref{pro:global-existence-v}
to obtain the assertion in~\eqref{estimate-d3}.
\end{proof}
\noindent
(c) Suppose $d>2$ and $q\ge d$. Then  \eqref{estimate-d3} holds true with $\tau=0$ for initial values 
$v_0\in B^{2\mu-2/p}_{qp,\sigma}(\Sigma, {\sf T}\Sigma)$ satisfying the assumptions of  \cite[Theorem 3.5(b)]{PSW20} and 
$(v_0 | z)_\Sigma =0$ for all $z\in\cE$.

\medskip\noindent
(d)
Suppose $p>2$ and $q\ge d>2$. 
Then every {\bf global} solution of \eqref{eq:weakSurfaceNS}, respectively \eqref{Abstract-Navier-Stokes},
converges exponentially fast to an equilibrium, namely to $P_\cE u_0$ (where $P_\cE$ is the projection defined in~\eqref{projection}), 
provided  the initial value $u_0$ satisfies the assumptions of Theorem~\ref{thm:weakSrurfaceStokes}  or  \cite[Theorem 3.5(b)]{PSW20}.
\begin{proof}
This follows by similar arguments as in the proof of Remark~\ref{pro:v-linearized} and Theorem~\ref{thm:convergence}.
\end{proof}
\medskip\noindent
(e)
According to Theorem 4.3 in \cite{PSW20}, any solution with initial value $u_0$ sufficiently close an equilibrium exists globally 
and, hence,  converges to $P_\cE u_0$ at an exponential rate.
\end{remarks}

\bigskip
%%%%%%%%%%%%%%%%%%%%%%%
\appendix
\section{}
%%%%%%%%%%%%%%%%%%%%%

\subsection{Auxilliary results}

We first consider a weak elliptic problem on compact manifolds without boundary.

\noindent
{
\begin{lemma}\label{lem:auxellprb1}
Let $1<q<\infty$. For each $v\in L_q(\Sigma,\mathsf{T}\Sigma)$ there exists a unique solution $\nabla_\Sigma\psi\in L_q(\Sigma,\mathsf{T}\Sigma)$ of
$$(\nabla_\Sigma\psi|\nabla_\Sigma\phi)_\Sigma=(v|\nabla_\Sigma\phi)_\Sigma,\quad \phi\in\dot{H}_{q'}^1(\Sigma).$$
\end{lemma}
\begin{proof}
For $\lambda>0$, let $A_0:=\lambda-\Delta_\Sigma$ in $X_0=L_{q}(\Sigma)$ with domain $X_1=H_q^2(\Sigma)$. By \cite[Theorem V.1.5.1]{Ama95}, the pair $(X_0,A_0)$ generates an interpolation-extrapolation scale $(X_\alpha,A_\alpha)$, $\alpha\in\R$, with respect to the complex interpolation functor $[\cdot,\cdot]_\theta$, $\theta\in (0,1)$. Let $X_0^\sharp=(X_0)'=L_{q'}(\Sigma)$ and denote by $A_0^\sharp$ the dual operator of $A_0$ in $X_0^\sharp$ with domain $X_1^\sharp:=H_{q'}^2(\Sigma)$. We write $(X_\alpha^\sharp,A_\alpha^\sharp)$, $\alpha\in\R$, for the dual interpolation-extrapolation scale generated by $(X_0^\sharp,A_0^\sharp)$. Then, $A_{-1/2}:X_{1/2}\to X_{-1/2}$ is a linear isomorphism,
where
$$X_{1/2}=[X_0,X_1]_{1/2}=H_q^1(\Sigma)$$
and
$$X_{-1/2}=\left(X_{1/2}^\sharp\right)'=\left([X_0^\sharp,X_1^\sharp]_{1/2}\right)'=\left(H_{q'}^1(\Sigma)\right)'=:H_q^{-1}(\Sigma).$$
We claim
$$\la A_{-1/2} u,\phi\ra=\int_\Sigma (\nabla_\Sigma u | \nabla_\Sigma\phi)\, d\Sigma+\lambda(u|\phi)_\Sigma,$$
for all $(u,\phi)\in X_{1/2}\times X_{1/2}^\sharp$. Indeed, for $u\in X_1$, it holds that $A_{-1/2}u=A_0 u$, hence
$$\la A_{-1/2} u,\phi\ra=(A_0u|\phi)_\Sigma=-(\Delta_\Sigma u|\phi)_\Sigma+\lambda(u|\phi)_\Sigma.$$
The surface divergence theorem as well as the density of $X_1$ in $X_{1/2}$ yield the claim.

Define an operator $B:X_{1/2}\to X_{-1/2}$ by $Bu=A_{-1/2}u-\lambda u$. Since the embedding $X_{1/2}\hookrightarrow X_{-1/2}$ is compact, the spectrum $\sigma(B)$ consists solely of eigenvalues with finite multiplicity. Furthermore, $\sigma(B)$ is invariant with respect to $q$ and for each eigenfunction $u$ of $B$ it holds that $u\in H_r^1(\Sigma)$ for \emph{any} $r\in (1,\infty)$.

We show that $\mu=0$ is a semi-simple eigenvalue of $B$. The equation $Bu=0$ in $X_{-1/2}$ is equivalent to
$$0=\int_\Sigma (\nabla_\Sigma u | \nabla_\Sigma\phi)\, d\Sigma$$
for all $\phi\in X_{1/2}^\sharp=H_{q'}^1(\Sigma)$. Choosing $\phi=u$, we obtain $\nabla_\Sigma u=0$, hence $u$ is constant. This shows
$$N(B)=\{u\in H_q^1(\Sigma)\mid u\ \text{is constant}\}.$$
We show $N(B^2)=N(B)$. For that purpose, let $u\in N(B^2)$ and define $v:=Bu$. Then $v\in N(B)$, hence $v$ is constant. For $\lambda>0$ we have
$$A_{-1/2}u=\lambda u+Bu=\lambda u+v\in L_q(\Sigma).$$
Solve $A_0w=\lambda u+v$ to obtain a unique solution $w\in H_q^2(\Sigma)$. Since $A_0 w=A_{-1/2} w$, this yields $A_{-1/2} u=A_{-1/2}w$ and therefore $u=w$ by injectivity of $A_{-1/2}$. This proves $u\in H_q^2(\Sigma)$ which in turn yields $A_0 u=\lambda u+v$ or equivalently $-\Delta_\Sigma u=v$. Integrating the last equation over $\Sigma$, yields $v=0$ as $v$ is constant. This shows $u\in N(B)$, hence $N(B^2)\subset N(B)$. Since the converse inclusion is obvious, we obtain the assertion.

We have shown so far that $\mu=0$ is a semi-simple eigenvalue of $B$. In particular, this implies
$$X_{-1/2}=N(B)\oplus R(B).$$
Consequently, the restricted operator $B:X_{1/2}\cap R(B)\to R(B)$ is invertible. Note that
$$R(B)=\{f\in X_{-1/2}\mid \la f,1\ra=0\}.$$
For $v\in L_q(\mathsf{T}\Sigma)$, define $f\in X_{-1/2}$ by $\la f,\phi\ra=(v|\nabla_\Sigma\phi)_\Sigma$. Then obviously $f\in R(B)$, hence there exists a unique solution $u\in X_{1/2}\cap R(B)$ of the equation $Bu=f$. The proof is completed.
\end{proof}
%%%%%%%%%%%%%%%%%%%%%
%%%%%%%%%%%%%%%%%%%%%
Next we study existence and uniqueness as well as regularity properties of solutions to some second order differential equations on $\R^d$.

For that purpose, we set
$$\la F_f,\phi\ra:=\int_{\R^d}f\phi\, dx,\quad \phi\in C_c^\infty(\R^d),$$
for  functions $f\in L_q(\R^d)$.
For $\lambda>0$, $k\in\{-1,0,1\}$ we then define
the function spaces
\begin{equation*}
\begin{aligned}
\mathbb{F}_\lambda^k:&=(\mathbb{F}^k,|\cdot|_{\mathbb{F}_\lambda^k}),\quad \mathbb{F}^k:=
\begin{cases}
\dot{H}_q^{-1}(\R^d),\ &k=-1,\\
\{f\in H_q^k(\R^d)\mid F_f\in \dot{H}_q^{-1}(\R^d)\},\ &k\in\{0,1\},
\end{cases} \\
\mathbb{E}_\lambda^k:&=(\mathbb{E}^k,|\cdot|_{\mathbb{E}_\lambda^k}),\quad \mathbb{E}^k:=\bigcap_{j=1}^{k+2}\dot{H}_q^j(\R^d),
\end{aligned}
\end{equation*}
%    $$\mathbb{F}_\lambda^k:=(\mathbb{F}^k,|\cdot|_{\mathbb{F}_\lambda^k}),\quad \mathbb{F}^k:=
%    \begin{cases}
%    \dot{H}_q^{-1}(\R^d),\ &k=-1,\\
%    \{f\in H_q^k(\R^d)\mid F_f\in \dot{H}_q^{-1}(\R^d)\},\ &k\in\{0,1\},
%    \end{cases}
%    $$
%    and
%    $$\mathbb{E}_\lambda^k:=(\mathbb{E}^k,|\cdot|_{\mathbb{E}_\lambda^k}),\quad \mathbb{E}^k:=\bigcap_{j=1}^{k+2}\dot{H}_q^j(\R^d)$$
equipped with the parameter-dependent norms
\begin{equation*}
\begin{aligned}
|f|_{\mathbb{F}_\lambda^{-1}} : &=|f|_{\dot{H}_q^{-1}(\R^d)}, && \\
|f|_{\mathbb{F}_\lambda^k}: &=
\sum_{j=0}^k\lambda^{(k-j)/2}|\nabla^jf|_{L_q(\R^d)}+\lambda^{(k+1)/2}|F_f|_{\dot{H}_q^{-1}(\R^d)},
&& k\in\{0,1\} \\
|u|_{\mathbb{E}_\lambda^k}:&=\sum_{j=1}^{k+2}\lambda^{(k+2-j)/2}|\nabla^j u|_{L_q(\R^d)}, &&k\in\{-1,0,1\}. &&
\end{aligned}
\end{equation*}
We are ready to prove the following result.}
%%%%%%%
\begin{lemma}\label{lem:auxellop1}
Let $1<q<\infty$, $k\in\{-1,0,1\}$ and $G\in W_\infty^{1+k}(\R^d)^{d\times d}$. Then there exist $\eta,\lambda_0>0$ such that
$$\operatorname{div}(G\nabla\cdot):\mathbb{E}_\lambda^k\to \mathbb{F}_\lambda^k$$
is an isomorphism, provided $|G-I|_{L_\infty(\R^d)}\le\eta$ and $\lambda\ge\lambda_0$.

Moreover, in this case there exists a constant $c(\lambda_0)>0$ such that the unique solution of
${\rm div}\,(G\nabla u)=f$ satisfies the estimate
\begin{equation}
\label{norm-estimate}
|\nabla u|_{H_q^{k+1}(\R^d)}\le c(\lambda_0) |f|_{\mathbb F^k_\lambda},\quad k\in\{-1,0,1\},\quad \lambda\ge \lambda_0,
\end{equation}
for any $f\in \mathbb F^k_\lambda$.
\end{lemma}
\begin{proof}
(i)
We start with the case $k=-1$. For $u\in\mathbb{E}_\lambda^{-1}$ we write
$$\la \operatorname{div}\,(G\nabla u),\phi\ra=\la\Delta u,\phi\ra+\la\operatorname{div}((G-I)\nabla u),\phi\ra,$$
where
$$\la \operatorname{div} v,\phi\ra:=-\int_{\R^d}v\cdot \nabla \phi\, dx$$
for $(v,\phi)\in L_q(\R^d)^d\times \dot{H}_{q'}^1(\R^d)$.

It is known that the operator $\Delta:\mathbb{E}_\lambda^{-1}\to \mathbb{F}_\lambda^{-1}$
is an isomorphism,  see for instance \cite{Tri83}, Theorem 5.2.3.1(i) and the remarks in Section 5.2.5 concerning duality,
or \cite[Lemma 3.3]{SiSo92}.
Furthermore, the estimate
$$|\la\operatorname{div}((G-I)\nabla u),\phi\ra|\le \eta|u|_{\dot{H}_q^1(\R^d)}|\phi|_{\dot{H}_{q'}^1(\R^d)}$$
holds. A Neumann series argument yields that the operator ${\operatorname{div}(G\nabla\cdot):\mathbb{E}_\lambda^{-1}\to \mathbb{F}_\lambda^{-1}}$
is an isomorphism as well, provided $\eta\in (0,1)$.

\medskip\noindent
(ii)
Let $k=0$. Then we have
$$\operatorname{div}(G\nabla u)=\Delta u+\operatorname{div}((G-I)\nabla u).$$
Also in this case, $\Delta:\mathbb{E}_\lambda^0\to\mathbb{F}_\lambda^0$ is an isomorphism.
Furthermore, there exists a constant $C>0$ such that
$|\Delta^{-1}f|_{\mathbb{E}_\lambda^0}\le C|f|_{\mathbb{F}_\lambda^0}$
for all $f\in {\mathbb{F}_\lambda^0}$ and all $\lambda\in (0,\infty)$. To see this, note that
$$|\nabla(\Delta^{-1}F_f)|_{L_q(\R^d)}\le C|F_f|_{\dot{H}_q^{-1}(\R^d)}$$
and
$$|\nabla^2(\Delta^{-1}f)|_{L_q(\R^d)}\le C|f|_{L_q(\R^d)}.$$

The definition of the norms in $\mathbb{E}_\lambda^0$ and $\mathbb{F}_\lambda^0$ yields the claim.
We then estimate as follows
\begin{align*}
|\operatorname{div}((G-I)\nabla u)|_{L_q(\R^d)}&\le |G|_{W_\infty^1(\R^d)}|\nabla u|_{L_q(\R^d)}+|G-I|_{L_\infty(\R^d)}|\nabla^2 u|_{L_q(\R^d)}\\
&\le C(\lambda^{-1/2}+\eta)|u|_{\mathbb{E}_{\lambda}^0}.
\end{align*}
Furthermore, as in the case $k=-1$, we have
$$\lambda^{1/2}|\operatorname{div}((G-I)\nabla u)|_{\dot{H}_q^{-1}(\R^d)}\le\eta|u|_{\mathbb{E}_\lambda^0}.$$
Once again, a Neumann series argument shows that $\operatorname{div}(G\nabla\cdot):\mathbb{E}_\lambda^0\to\mathbb{F}_\lambda^0$ is an isomorphism, provided $\eta$ and $\lambda^{-1/2}$ are chosen sufficiently small.

\medskip\noindent
(iii)
The proof for the remaining case $k=1$ follows literally the same strategy and is therefore omitted.

\medskip\noindent
(iv) The assertion in \eqref{norm-estimate} follows from the steps above and the definition of the norm of $\mathbb E^k_\lambda$.
\end{proof}

\subsection{Korn's inequality}

We will establish an appropriate version of Korn's inequality for compact surfaces $\Sigma$ without boundary.
For this, we use the notation from Section \ref{subsec:EnergyEst}.

%%%%%%%%%%%
\begin{theorem}[Korn's inequality]
\label{thm:Korn}
There exists a constant $C>0$ such that
\begin{equation}
\label{Korn}
|v|_{H^1_2(\Sigma)} \le C |\mathcal D_\Sigma(v)|_{L_2(\Sigma)}\quad\text{for all}\ v\in V^1_2(\Sigma).
\end{equation}
\end{theorem}
%%%%%%%%%%
\begin{proof}
Let $Lu:=2\cP_\Sigma{\rm div}_\Sigma \mathcal D_\Sigma(u) $ for $u\in H^2_{2,\sigma}(\Sigma,\mathsf{T}\Sigma)$.
We know from Proposition A.2 in \cite{PSW20} that
$$Lu=\Delta_\Sigma u + {\sf Ric}_\Sigma\,u,$$
where $\Delta_\Sigma$ is the Bochner (the connection) Laplacian and ${\sf Ric}_\Sigma\,u$ is the Ricci $(1,1)$-tensor, given in local coordinates by
$({\sf Ric}_\Sigma)^\ell_k= g^{i\ell}R_{ik}$, so that  $ {\sf Ric}_\Sigma\,u =  R^\ell_k u^k\tau_\ell.$
%$$ {\sf Ric}_\Sigma^\ell_k= g^{i\ell}R_{ik}\quad\text{so that}\quad g^{\#} {\sf Ric}_\Sigma\,u =  R^\ell_k u^k\tau_\ell .$$
It is well-known that $(\Delta_\Sigma u|u)_\Sigma = - |\nabla u|^2_{L_2(\Sigma)}$ for $u\in H^2_2(\Sigma,{\sf T}\Sigma)$,
see for instance \cite[Lemma 3.5]{SaTu20}.
Let  $u\in H^2_{2,\sigma}(\Sigma,\mathsf{T}\Sigma)$. Then
\begin{equation*}
\begin{aligned}
2\int_\Sigma |\mathcal D_\Sigma(u)|^2\,d\Sigma &=-\int_\Sigma (Lu|u)\,d\Sigma
= -\int_\Sigma \Big((\Delta_\Sigma u | u)+({\sf Ric}_\Sigma\,u | u)\Big)\,d\Sigma \\
&  =\int_\Sigma |\nabla u|^2\,d\Sigma  -\int_\Sigma ({\sf Ric}_\Sigma\,u | u)\,d\Sigma
\ge |\nabla u|^2_{L_2(\Sigma)} - C_1 |u|^2_{L_2(\Sigma)}.
\end{aligned}
\end{equation*}
Here we used that on the compact surface $\Sigma$, the Ricci tensor $ {\sf Ric}_\Sigma$ can be bounded uniformly, yielding
%For the last estimate we employed compactness of $\Sigma$ to conclude that
$$ \int_\Sigma ({\sf Ric}_\Sigma\,u | u)\,d\Sigma \le C_1 |u |^2_{L_2(\Sigma)}$$
with an appropriate positive constant $C_1$.
By density of the space $H^2_{2,\sigma}(\Sigma,\mathsf{T}\Sigma)$ in  $H^1_{2,\sigma}(\Sigma,\mathsf{T}\Sigma)$
we obtain
\begin{equation}
\begin{aligned}
\label{Korn-1}
  |u|^2_{H^1_2(\Sigma)} \le C_2 \Big( |\mathcal D_\Sigma (u)|^2_{L_2(\Sigma)} + | u |^2_{L_2(\Sigma)}\Big),\quad u\in H^1_{2,\sigma}(\Sigma,\mathsf{T}\Sigma),
\end{aligned}
\end{equation}
for an appropriate constant $C_2$.
The assertion in \eqref{Korn} now follows from the Petree-Tartar Lemma. For the readers' convenience, we include the proof here.

Suppose  \eqref{Korn} does not hold.
Then there exists a sequence $(v_n)_{n\in\N}$ in $V^1_2(\Sigma)$, see \eqref{V2}, such that
$| v_n |_{H^1_2(\Sigma)}=1$ for $n\in\N$ and $|\mathcal D_\Sigma (v_n)|_{L_2(\Sigma)}\to 0\ \text{as}\ n\to \infty.$
% $\mathcal D_\Sigma(v_n)\to 0$ in $L_2$ as $n\to\infty$.
Since $H^1_2(\Sigma,\mathsf{T}\Sigma)$ is compactly embedded in $L_2(\Sigma,\mathsf{T}\Sigma)$, there exists a subsequence, still denoted by $(v_n)_{n\in\N}$,
which converges to an element $v\in L_2(\Sigma,\mathsf{T}\Sigma)$. From \eqref{Korn-1} follows that $ (v_n)_{n\in\N}$ is a Cauchy sequence in $V^1_2(\Sigma)$.
Completeness of $V^1_2(\Sigma)$ shows  that $v\in V^1_2(\Sigma)$ and $v_n\to v$ in $V^1_2(\Sigma)$.
Consequently, $ \mathcal D_\Sigma (v_n)\to \mathcal D_\Sigma(v)$ as well as  $ \mathcal D_\Sigma (v_n) \to 0$ in $L_2(\Sigma, {\sf T}^1_1 \Sigma )$.
This implies $\mathcal D_\Sigma (v)=0$, and then $v\in \mathcal E\cap  V^1_2(\Sigma)=\{0\}$,
in contradiction to   the assumption $|v|_{H^1_2(\Sigma)}=1$.
\end{proof}
%%%%%%%%%%
\begin{remark}
\label{remark:Korn}
 Korn's inequality for embedded surfaces has also been established in \cite[Lemma 4.1]{JOR18}.
The proof given here is considerably shorter.
\end{remark}

%%%%%%%%%%%%%%%%%%


\begin{thebibliography}{99}
\frenchspacing

\bibitem{Abe02} H. Abels, 
Boundedness of imaginary powers of the Stokes operator in an infinite layer. 
{\em J. Evol. Equ.}{\bf 2} (2002), 439-457.

\bibitem{Ama95}
H.~Amann,
{\em Linear and Quasilinear Parabolic Problems}, {Monographs in Mathematics {\bf 89}},
 Birkh\"auser 1995.
 
 \bibitem{Ama13} H. Amann,
 Function spaces on singular manifolds. 
 {\em Math. Nachr.} {\bf 286} (2013), 436-475.

%\bibitem{ArCr12}  M. Arnaudon,   A.B. Cruzeiro,
%Lagrangian {Navier--Stokes} diffusions on manifolds: variational principle and stability.
%{\em Bulletin des Sciences Math{\'e}matiques} {\bf 136}  (2012), 857--881.

\bibitem{BeLo76} J. Bergh, J. L\"ofstr\"om, 
{\em Interpolation spaces. An introduction.} 
Grundlehren der Mathematischen Wissenschaften, No. 223. Springer-Verlag, Berlin-New York, 1976. 

\bibitem{BoPr10} D. Bothe, J. Pr\"uss,
On the two-phase Navier-Stokes equations with Boussinesq-Scriven surface fluid.
{\em J. Math. Fluid Mech.} {\bf 12} (2010),133-150.

%\bibitem{Bou13}  J. Boussinesq,
%Sur l\'{e}xistence d`une viscosit\'{e} superficielle, dans la mince couche de
%transition s\'{e}parant un liquide d'une autre fluide contigu. {\em Ann. Chim. Phys.} {\bf 29} (1913), 349--357.

\bibitem{CCD17} C.H. Chan, M. Czubak, M. Disconzi,
The formulation of the Navier-Stokes equations on Riemannian manifolds.
{\em J. Geom. Phys.} {\bf 121} (2017), 335-346.

\bibitem{DDHPV04} R. Denk, G. Dore, M. Hieber, J. Pr\"{u}ss,  A. Venni,
New thoughts on old results of R.T. Seeley.
{\em Math. Ann.} {\bf 328} (2004), 545-583.

\bibitem{EbMa70} D.G. Ebin, J. Marsden,
Groups of diffeomorphisms and the motion of an incompressible fluid.
 {\em Annals of Mathematics (2)} {\bf 92} (1970), 102-163.
 
 \bibitem{FuKa62} 
H.~Fujita and T.~Kato.
On the non-stationary Navier-Stokes system.
{\em Rend.~Sem.~Mat., Univ.~Padova}, {\bf 32} (1962), 243-260.

\bibitem{Gig85} Y. Giga.
 Domains of fractional powers of the Stokes operator in $L_r$ spaces.
 {\em  Arch. Rational Mech. Anal.} {\bf 89}, (1985), 251-265.

\bibitem{HiSa18} M.~Hieber, J.~Saal,
The Stokes equation in the $L^p$-setting: well-posedness and regularity properties.
In Handbook of mathematical analysis in mechanics of viscous fluids, 117-206, Springer, Cham, 2018.

\bibitem{JOR18} T. Jankuhn, M.A. Olshanskii, A. Reusken,
Incompressible fluid problems on embedded surfaces: modeling and variational formulations.
{\em  Interfaces Free Bound.} {\bf 20} (2018), 353-377.

\bibitem{Kat84} T. Kato.
Strong $L^p$-solutions of the Navier-Stokes equation in $\R^m$, with applications to weak solutions. 
{\em Math. Z.} {\bf 187} (1984), 471-480.

\bibitem{KLG17} H. Koba, C. Liu, Y. Giga,
 Energetic variational approaches for incompressible fluid systems on an evolving surface.
 {\em Quart. Appl. Math.} {\bf 75} (2017), 359-389.

%\bibitem{LPW14} J.~LeCrone, J.~Pr{\"u}ss, and M.~Wilke.
% On quasilinear parabolic evolution equations in weighted  {$L_p$}-spaces {II}.
% {\em J. Evol. Equ.}, 14(3), 509--533, 2014.

 \bibitem{MPS19} G. Mazzone, J. Pr\"uss, G. Simonett,
 A maximal regularity approach to the study of motion of a rigid body with a fluid-filled cavity.
 {\em J. Math. Fluid Mech.} {\bf 21} (2019), no. 3, Paper No. 44, 20 pp.

 \bibitem{Maz03} A. Mazzucato,
 Besov-Morrey spaces: function space theory and applications to non-linear PDE.
 {\em Trans. Amer. Math. Soc.} {\bf 355} (2003), 1297-1364.
 
 \bibitem{NoSa03} A. Noll, J. Saal, 
 $H^\infty$-calculus for the Stokes operator on $L_q$-spaces. 
 {\em Math. Z.} {\bf 244}, (2003), 651-688. 

\bibitem{OQRY18}
M.A. Olshanskii, A. Quaini,  A. Reusken, V. Yushutin,
A finite element method for the surface Stokes problem.
{\em SIAM J. Sci. Comput.} {\bf 40} (2018), A2492--A2518.

\bibitem{Pri94} V. Priebe,
Solvability of the Navier-Stokes equations on manifolds with boundary.
{\em Manuscripta Math.} {\bf 83} (1994), no. 2, 145--159

\bibitem{Pru18} J.~Pr\"uss,
$H^\infty$-calculus for generalized Stokes operators.
{\em J. Evol. Eq.}  {\bf 18} (2018), 1543-1574.

\bibitem{PrSi04}  J.~Pr\"uss, G. Simonett, 
Maximal regularity for evolution equations in weighted $L_p$-spaces. 
{\em Arch. Math. (Basel)} {\bf 82} (2004), 415-431. 


\bibitem{PrSi16} J.~Pr\"uss, G.~Simonett,
{\em Moving Interfaces and Quasilinear Parabolic Evolution Equations}. Monographs in Mathematics {\bf 105}, Birkh\"auser  2016.

\bibitem{PSW18} J.~Pr\"uss, G.~Simonett, M.~Wilke,
Critical spaces for quasilinear parabolic evolution equations and applications.
{\em J. Differential Equations}  {\bf 264} (2018), 2028-2074.

\bibitem{PSW20} J.~Pr\"uss, G.~Simonett M.~Wilke,
On the Navier-Stokes equations on surfaces.
{\em J. Evol. Equ.} (2020). https://doi.org/10.1007/s00028-020-00648-0.

\bibitem{PrWi17} J.~Pr\"uss, M.~Wilke,
Addendum to the paper "On quasilinear parabolic evolution equations in weighted $L_p$-spaces II'',
{\em  J. Evol. Equ.} {\bf 17} (2017), 1381-1388.

\bibitem{PrWi18} J.~Pr\"uss, M.~Wilke,
On critical spaces for the Navier-Stokes equations.
{\em J. Math. Fluid Mech.} {\bf 20} (2018), 733-755.

%\bibitem{PrWi19} J.~Pr\"uss, M.~Wilke,
%{\em Gew\"ohnliche Differentialgleichungen und dynamische Systeme}. Grundstudium Mathematik. Birkh\"auser Verlag, Basel, 2019, 2nd edition.

\bibitem{ReZh13} A. Reusken,Y. Zhang,
Numerical simulation of incompressible two-phase flows with a Boussinesq-Scriven interface stress tensor.
{\em Internat. J. Numer. Methods Fluids}  {\bf 7}  (2013), 1042--1058.

\bibitem{ReVo18} S. Reuther, A. Voigt,
Solving the incompressible surface Navier-Stokes equation by surface elements,
{\em Phys. Fluids}  {\bf 30} (2018), 012107. %arxiv 1709.02803.

\bibitem{Sal06} J. Saal, 
Stokes and Navier-Stokes equations with Robin boundary conditions in a half space. 
{\em J. Math. Fluid Mech.} {\bf 8} (2006), 211-241.

%\bibitem{Sak96} T. Sakai,
%{\em Riemannian geometry}. Translations of Mathematical Monographs, 149. American Mathematical Society, Providence, RI, 1996.

\bibitem{SaTu20} M. Samavaki, J. Tuomela, 
Navier-Stokes equations on Riemannian manifolds. 
{\em J. Geom. Phys.} {\bf 148} (2020), 103543, 15 pp.

%\bibitem{Sam68} H. Samelson,
%Orientability of hypersurfaces in $\R^n$.
%{\em Proc. Amer. Math. Soc.}  \textbf{22},  301--302 (1969).

%\bibitem{Scr60} L. E. Scriven,
%Dynamics of a fluid interface, {\em Chem. Eng. Sci.} {\bf 12} (1960), 98--108.

\bibitem{SiSo92} C.G. Simader, H. Sohr,
A new approach to the Helmholtz decomposition and the Neumann problem in $L_q$-spaces for bounded and exterior domains.
In {\em Mathematical problems relating to the Navier-Stokes equation}, Ser. Adv. Math. Appl. Sci., 11,
World Sci. Publ., River Edge, NJ, 1992, 1--35.

%\bibitem{SSO07} J. C. Slattery, L. Sagis, E.-S. Oh,
%{\em Interfacial Transport Phenomena}. 2nd ed. Springer, New York, 2007.

\bibitem{Tay92} M. E. Taylor,
Analysis on Morrey spaces and applications to Navier-Stokes and other evolution equations.
{\em Comm. Partial Differential Equations} {\bf 17} (1992), 1407-1456.

%{\color{blue}
%\bibitem{Tem88} R. Temam,
%{\em Infinite-dimensional dynamical systems in mechanics and physics}.
%\newblock Springer, New York, 1988.}

\bibitem{Tri78}
H.~Triebel.
\newblock {\em Interpolation theory, function spaces, differential operators}.
\newblock North-Holland Mathematical Library, 18. North-Holland Publishing Co., Amsterdam-New York, 1978.

\bibitem{Tri83} H. Triebel,
{\em Theory of function spaces}.
Monographs in Mathematics, {\bf 78},  Birkh\"auser Verlag, Basel, 1983.

\end{thebibliography}
\end{document}